\documentclass[11pt]{amsart}

\usepackage{amssymb,amsmath,enumerate,epsfig}%,epsf}%t1enc
\usepackage{amsfonts,color}
\usepackage{a4,latexsym}
\usepackage{parskip}
\usepackage[all]{xy}
\usepackage{hyperref}
\hypersetup{pdftitle=dyn}

\newcommand{\C} {\mathbb{C}}
\newcommand{\Q} {\mathbb{Q}}
\newcommand{\N}  {\mathbb{N}}

\newcommand{\F}{\mathbb{F}}
\newcommand{\Z}{\mathbb{Z}}

\newcommand{\PP}{\mathbb{P}}
\newcommand{\NS}{\mathop{\rm NS}}
\newcommand{\Num}{\mathop{\rm Num}}
\newcommand{\MW}{\mathop{\rm MW}}
\newcommand{\MWL}{\mathop{\rm MWL}}

\newcommand{\disc}{\mathop{\rm disc}}

\newcommand{\Pic}{\mathop{\rm Pic}}

\newcommand{\mX}{\mathcal X}

\newcommand{\mR}{\mathcal R}

\newcommand{\Jac}{\mathop{\rm Jac}}

\newcommand{\tx}{\mathrm x}
\newcommand{\ty}{\mathrm y}

\newtheorem{Theorem}{Theorem}[section]
\newtheorem{Proposition}[Theorem]{Proposition}
\newtheorem{Lemma}[Theorem]{Lemma}

\newtheorem{Corollary}[Theorem]{Corollary}

\theoremstyle{remark}

\newtheorem{Remark}[Theorem]{Remark}
\newtheorem{Criterion}[Theorem]{Criterion}
\newtheorem{Example}[Theorem]{Example}

\theoremstyle{definition}

\newtheorem{Definition}[Theorem]{Definition}
\newtheorem{Observation}[Theorem]{Observation}

\newtheorem{Convention}[Theorem]{Convention}
\newtheorem{claim}[Theorem]{Claim}

\begin{document}

\title[$\Q_\ell$-cohomology projective planes]{$\Q_\ell$-cohomology projective planes from Enriques surfaces in odd characteristic}

%\dedicatory{Dedicated to Tetsuji Shioda on the occasion of his 75th birthday}

\author{Matthias Sch\"utt}
\address{Institut f\"ur Algebraische Geometrie, Leibniz Universit\"at
  Hannover, Welfengarten 1, 30167 Hannover, Germany}
\email{schuett@math.uni-hannover.de}

%%
%\keywords{K3 surface, wild automorphism, Lefschetz fixed point formula}
%%
%%
\thanks{Funding   by ERC StG~279723 (SURFARI) 
 is gratefully acknowledged.}

\date{\today}

\begin{abstract}
We give a complete classification of  $\Q_\ell$-cohomology  projective planes
with isolated ADE-singularities and numerically trivial canonical bundle in odd characteristic.
This leads to a beautiful relation with certain Enriques surfaces
which  parallels the situation in characteristic zero, yet displays intriguing subtleties.
\end{abstract}
%
%
% \begin{abstract}
%This paper concerns K3 surfaces with automorphisms of order 11 in arbitrary characteristic.
%Specifically we study the wild case and prove that a general such surface in characteristic 11
%has Picard number 2.
%We also construct K3 surfaces with an automorphism of order 11 in every characteristic,
%and supersingular K3 surfaces whenever possible.
% \end{abstract}
%% 
 \maketitle

 \section{Introduction}
 \label{s:intro}

 Over the complex numbers,
 there is a complete understanding of 
 fake projective planes,
 i.e.~smooth projective surfaces
 whose singular homology resembles that of $\PP^2$,
 with Betti numbers $b_2=1, b_1=b_3=0$,
 while not being isomorphic to $\PP^2$ itself.
 Necessarily, a fake projective plane has general type;
this offers one explanation
why the classification (due to Prasad--Yeung \cite{PY} and Cartwright--Steger \cite{CS})  remains group-theoretic in nature.
In fact, our understanding of the geometry of fake projective planes
is rather rudimentary,
and in particular, there are no explicit models known
(in terms of equations as opposed to abstract ball quotients).

 Another downside of the above picture is
 that it does not carry over to the positive characteristic side,
 although also there the notion of a fake projective plane $S$ makes perfect sense, 
 with $\ell$-adic cohomology replacing singular homology.
 It may come surprising at first that this can be helped
 by allowing for $S$ to be singular.
 Here we shall restrict 
 to isolated ordinary double points as singularities,
 partly because these do not affect the canonical divisor $K_S$ (or the dualizing sheaf).
 We then distinguish whether $K_S$ is positive, numerically trivial or negative.
This paper concentrates on the case which is rich, yet accessible in nature,
where $K_S$ is numerically trivial.
In contrast to the situation over $\C$ which was treated in \cite{S-Q-hom},
this leads to two possibilities for the minimal desingularization
\begin{eqnarray}
\label{eq:Y}
Y \to S.
\end{eqnarray}
Either $Y$ is a K3 surface or an Enriques surface.
Since the former case only occurs in small characteristic by \cite{Shimada-rk21},
we shall investigate the latter case,
i.e.~$Y$ is an Enriques surface supporting a root type $R$ of rank $9$
corresponding to the configuration of exceptional curves in \eqref{eq:Y}.
Since Enriques surfaces behave rather differently in characteristic $2$
(although many arguments can be adapted for singular Enriques surfaces),
we will deal with them separately in a different paper \cite{S-2} 
and consider only fields of odd characteristic here.

\begin{Theorem}
\label{thm}
There are 31 root types of rank $9$ realized through smooth rational curves on Enriques surfaces
in odd characteristic (28 in characteristic $3$, 30 in characteristic $5$).
For each the following hold:
\begin{enumerate}[(i)]
\item
the root types are supported on 1-dimensional families of Enriques surfaces;
\item
the moduli spaces  can have up to 3 different components;
\item
each family has rational base and is defined over the prime field;
\item
each family can be parametrized explicitly,
see Tables \ref{T2'} and \ref{T2}.
\end{enumerate}
\end{Theorem} 

%[extend to odd char: exclude 1 type for $p=3$ (?); \\
%check parametrizations: factor 27?\\
%start over alg closed field, then family
%def'd over prime field]
%
%This also corrects an error from \cite{Keum-Viet}.
%

It is intriguing how Theorem \ref{thm} parallels the complete classification
of these Enriques surfaces over $\C$.
In fact, we will derive explicit parametrizations of all Enriques surfaces in Theorem \ref{thm}
which do not depend on the characteristic at all (except for the three root types
which cease to exist in characteristic $3$, and one in characteristic $5$).
To arrive at the explicit parametrizations, however,
we have to be extremely careful because many of the arguments over $\C$ do not carry over directly,
for instance those from \cite{HK} invoking the orbifold Bogomolov--Miyaoka--Yau inequality
and those based on \cite{Kondo-Enriques} using Nikulin's work on automorphisms on complex K3 surfaces.
In consequence, we will sometimes take shortcuts and refer the reader for the detailed arguments in \cite{S-Q-hom},
but this will be compensated by the detours which we have to take where the characteristic zero arguments
do not work or are not sufficient.

The paper is organized as follows.
After setting up the scenery in the Section \ref{s:set},
Section \ref{s:latt} reviews  root types on Enriques surfaces and the underlying lattice theory
relevant for this paper.
In Section \ref{s:ell}, we relate this to elliptic fibrations and rational elliptic surfaces.
This will already enable us to show that fields of positive characteristic admit no other root types than characteristic zero
(Proposition \ref{lem:no-add}).
The main geometric technique
of Enriques involutions of base change type is introduced  in Section \ref{s:bc}
(with a view towards positive characteristic).
We work out an explicit formulation, geared towards our aims, in Section \ref{s:eqns}
which then calls for a review of the key ingredients, namely extremal rational elliptic surfaces in Section \ref{s:res}.
This will put us in the position to develop explicit parametrizations on all Enriques surfaces in question
in Section \ref{s:fam} 
from the information obtained thus far (Tables \ref{T2'} \& \ref{T2}).
This will prove most part of Theorem \ref{thm}.
The final section concerns the number of components of the moduli spaces,
thus completing the proof of Theorem \ref{thm}.
Throughout the paper, we will work out instructive examples in detail
to give a clear idea of the methods and techniques involved.

\section{Set-up}
\label{s:set}

We work over an algebraically closed field $K$ of characteristic $p>2$.
Let $S$ be a normal projective surface over $K$
with only ADE singularities, i.e.~isolated rational double points.
For a prime $\ell\neq p$, let $b_i(S)$ be the Betti numbers of $S$ for  the $\ell$-adic \'etale cohomology:
\[
b_i(S) = \dim \mathrm{H}^i_{\text{\'et}}(S,\Q_\ell)
\]

\begin{Definition}
We call $S$ a \emph{$\Q_\ell$-cohomology projective plane}
if 
\[
b_1(S)=b_3(S)=0, \;\; b_2(S)=1.
\]
\end{Definition}

We emphasize that, except for the notation, the above notion does not depend on the chosen prime $\ell\neq p$ 
by the usual comparison theorems. 

Let $Y$ denote the minimal desingularization of $S$ as in \eqref{eq:Y}.
Due to the mild singularities, the canonical divisors of $Y$ and $S$ are the same,
and we distinguish whether $K_Y$ (or $K_S$) is positive, numerically trivial, or negative.
In the first case, $Y$ is of general type, and it seems hard to say anything about $Y$ even over $\C$.
Meanwhile in the last case, $Y$ is a del Pezzo surface,
and the complex classification from \cite{Ye} should carry over with a few modifications and exceptions,
yielding some  isolated surfaces and a one-dimensional family.

In the sequel, we assume that $K_Y$ is numerically trivial.
Since the vanishing $b_1(Y)=0$ is inherited from $S$,
the Enriques--Kodaira classification 
(extended to positive characteristic by Bombieri and Mumford \cite{BM1}, \cite{BM2}, \cite{Mumford-surf}) dictates
that $Y$ is either a K3 surface or an Enriques surface.
In either case, $Y$ supports a root type $R$ of maximal rank $b_2(Y)-1 = 21$ resp.~$9$ supported on smooth rational curves
by assumption.
In particular, all of $\mathrm H^2_{\text{\'et}}(Y,\Q_\ell)$ is algebraic,
i.e.~$Y$ is supersingular.
The root type $R$ will be a direct sum of root lattices $A_n, D_k, E_l$ which are all taken to be negative-definite
throughout this paper.
Orthogonal sums are indicated by a '$+$',
and likewise by $2R$ etc.

For K3 surfaces, this set-up turns out to be rather restrictive
(so restrictive, in fact, that they cannot occur over $\C$ by Lefschetz' $(1,1)$-theorem):

\begin{Theorem}[Shimada]
\label{thm:K3}
If $Y$ is a K3 surface supporting a root type of rank $21$,
then the characteristic satisfies $0<p\leq 19$.
\end{Theorem}

In fact, Shimada gives a much more precise result in  \cite{Shimada-rk21},
determining all possible root types starting from $A_{18}+A_3$ and $A_{18}+A_2+A_1$ in characteristic $19$
down to $21A_1$ in characteristic $2$.
It is interesting to see how the moduli dimensions go up in unison,
from zero in characteristics $19$ through $11$ to full 10-dimensional moduli in characteristic $2$.

The case of Enriques surfaces, in contrast, presents a totally different challenge
as it does not essentially depend on the characteristic and always leads to one-dimensional families
as recorded in Theorem \ref{thm}.
Partly, this may be explained because an Enriques surface $Y$ is always supersingular.
More precisely, we have
\[
H^2_{\text{\'et}}(Y,\Q_\ell) \cong \Pic(Y) \otimes\Q_\ell\cong \Q_\ell^{10}
\]
and, as a lattice,
\begin{eqnarray}
\label{eq:Num}
\Num(Y) = \Pic(Y)/(\text{torsion}) \cong U+E_8
\end{eqnarray}
by \cite{Illusie} where $U$ denotes the hyperbolic plane.
This lattice structure will play a crucial role throughout this paper,
both when classifying the possible root types and when working out the explicit parametrizations
of the Enriques surfaces proving Theorem \ref{thm}.

%
%\subsubsection*{Convention}
%
%All root lattices $A_n, D_k, E_l$ are taken to be negative-definite.
%Orthogonal sums are indicated by a '$+$',
%and likewise by $2R$ etc.
%The notation $R(m)$ indicates the same abstract $\Z$-module with intersection form multiplied by an integer $m$.

\section{Maximal root types on Enriques surfaces}
\label{s:latt}

It is an elementary exercise to compile a list of all root types $R$ of rank $9$,
starting from $A_9$ and $D_9$ all the way to $9A_1$
(most of which may be found in \ref{ss:31} and \ref{ss:add}).
Over $\C$, many of them can be ruled out to occur on Enriques surfaces
using the orbifold Bogomolov--Miyaoka--Yau inequality;
in fact, in \cite{HK} it is shown along these lines  %(and more generally for normal complex surfaces with quotient singularities)
that either $R$ has at most 4 orthogonal summands (and satisfies some extra conditions)
or $R=2A_3+3A_1$.

In positive characteristic, however, this approach is not available.
Here we shall therefore replace it with some general theory of quadratic forms
(which is completely characteristic free)
and special properties of elliptic fibrations on $Y$ which can be derived purely from $R$.

%and lifting arguments (which to a small extent are based on the explicit parametrizations
%which we will develop; in these cases they only apply a posteriori).
% -- the reason why characteristic $3$ turns out different from all others).

We start by reviewing the portion of the theory of quadratic forms which is relevant to our paper.
Here we could use Hilbert symbols and local epsilon-invariants following \cite{HKO},
but it will be more beneficial for our purposes to pursue an explicit approach based on the analysis of
discriminant forms \'a la Nikulin \cite{Nikulin}.
Let us explain how this applies to our precise problem.

Let $Y$ be an Enriques surface containing nine smooth rational curves
which support a given root type $R$ of rank $9$.
Then this gives an embedding
\begin{eqnarray}
\label{eq:hook}
R\hookrightarrow\Num(Y)
\end{eqnarray}
which naturally extends to the primitive closure
\[
R'= (R\otimes\Q)\cap\Num(Y) \hookrightarrow \Num(Y).
\]
For later use, we emphasize the following result which will 
be very useful for our purposes (proved in much larger generality in \cite{S-nodal}).

\begin{Lemma}
\label{lem:no}
If $R'\neq R$, then $R'$ is always obtained from $R$ by adding some $(-2)$-vectors,
but these can never be effective nor anti-effective.
\end{Lemma}

In particular, the lemma implies that the extra classes in $R'$ cannot be represented by smooth rational curves on $Y$.
This is contrary to the K3 case where Riemann--Roch applies to show that any $(-2)$-divisor
is either effective or anti-effective; on an Enriques surface, however, 
 the torsion in $\Pic(Y)$ prevents this implication from holding.

It follows from \eqref{eq:Num} and  \eqref{eq:hook} that there is a positive class $H\in\Num(Y)$ such that
\[
\Z H + R' \hookrightarrow \Num(Y)
\]
has finite index 
\[
d = H^2 = | \disc R' \, |.
\]
More precisely, the discriminant forms of the orthogonal summands have to be compatible,
i.e.~the discriminant group $(R')^\vee/R'$ together with the induced quadratic form with values in $\Q/2\Z$ and likewise for $H$
(or $\Z H$);
since $\Num(Y)$ is unimodular, this translates as
\begin{eqnarray}
\label{eq:q}
q_H = -q_{R'}.
\end{eqnarray}
In particular, we find that
\begin{eqnarray}
\label{eq:d}
(R')^\vee/R' \cong H^\vee/H \cong \Z/d\Z,
\end{eqnarray}
i.e.~$R'$ has length one (the minimum number of generators of the discriminant group)
which turns out to be a severe restriction on the root type $R$.
We illustrate this with a simple example.

\begin{Example}
\label{ex}
Let $R=D_4+A_4+A_1$. Then one readily checks that $R$ has no proper even integral overlattices,
i.e.~necessarily $R=R'$ if $R$ were to be supported on some Enriques surface $Y$,
but $R$ has length three, giving a contradiction by \eqref{eq:d}.

The same line of argument applies to $\hat R=A_4+A_1^5$ of length five,
since here the only  proper even integral overlattice is $R$.
\end{Example}

\begin{Example}[$R=A_5+2A_2$]
\label{ex1}
On the positive side, one checks that for $R=A_5+2A_2$ of length $3$,
the only even integral overlattice $R'=E_7+A_2$ has discriminant form matching $\Z H$ for $H^2=6$ up to sign,
so $R$ and $R'$ embed into $U+E_8$.
\end{Example}

It is easy to imagine (and check!) that this kind of argument suffices
to rule out a good number of root types (as was done in \cite{S-Q-hom}
which also provides many more explicit examples).
In particular, this applies to all root types with at most 4 summands (outside characteristics $3, 5$
where due to degenerations we have to throw in some extra work as shall see in \ref{ss} -- \ref{ss:5}).
However, without the orbifold Bogomolov--Miyaoka--Yau inequality at our disposal,
there is no obvious way to rule out even $9A_1$
in positive characteristic!
We shall work around this using an argument
which plays out elliptic fibrations against Lemma \ref{lem:no}.
The next section will review the construction.
Before we get there, we list the  root types which are not excluded by the quadratic forms argument sketched above.

\subsection{Root types from characteristic zero}
\label{ss:31}
%
%which we will lay out in detail later in Section \ref{s:lift}.
%
%For now, we content ourselves with stating the 31 possible root types:

By \cite{HKO} and \cite{S-Q-hom}, exactly the following root types of rank $9$ are supported on Enriques surfaces
in characteristic zero:

\begin{eqnarray*}
\label{eq:list} 
  A_9 ,  A_8 + A_1 ,  A_7 + A_2 ,  A_7 + 2A_1 ,   A_6 + A_2 + A_1 ,    A_5 + A_4 ,  A_5 + A_3 + A_1 , \nonumber \\ \nonumber
  A_5 + 2A_2 ,  A_5 + A_2 + 2A_1 ,  2A_4 + A_1 ,  A_4 + A_3 + 2A_1 ,  3A_3 ,   2A_3 + A_2+ A_1 , \\
 2A_3 + 3A_1 ,  A_3 + 3A_2 ,
 D_9 ,  D_8 + A_1 ,  D_7 + 2A_1 ,  D_6 + A_3 ,  D_6 + A_2 + A_1 , \nonumber \\ \nonumber
   D_6 + 3A_1 ,  D_5 + A_4 ,  D_5 + A_3 + A_1 ,  D_5 + D_4 ,   D_4 + A_3 + 2A_1 ,  2D_4 + A_1 ,\\
 E_8 + A_1 ,  E_7 + A_2 ,  E_7 + 2A_1 ,  E_6 + A_3 ,  E_6 + A_2 + A_1. 
\end{eqnarray*}

%\begin{Remark}
We highlight that Theorem \ref{thm} states that all these 31 root types indeed occur 
over fields of odd characteristic
-- except for some non-existence bit in characteristics $3$ and $5$ (see \ref{ss:3'}, \ref{ss:5}).
%\end{Remark}
%

\subsection{Additional root types}
\label{ss:add}

Here we list those 7 root types (with at least 5 orthogonal summands)
which admit an embedding into $U+E_8$,
but do not occur in characteristic zero due to the orbifold Bogomolow--Miyaoka--Yau inequality.
\begin{eqnarray*}
D_5+4A_1, D_4+5A_1, D_4+A_2+3A_1, % A_3+A_2+4A_1, \\
A_3+6A_1, 4A_2+A_1, A_2+7A_1, 9A_1.
\end{eqnarray*}

\section{Elliptic fibrations}
\label{s:ell}

We continue to elaborate on ideas from \cite{S-Q-hom} to sort out the possible maximal root types,
adjusted for our situation of positive characteristic.
Every Enriques surface $Y$
admits an elliptic fibrations
\begin{eqnarray}
\label{eq:ell}
f:\; Y \to \PP^1.
\end{eqnarray}
This has two ramified fibers, of multiplicity two,
so there cannot be a section 
(but there always is a bisection given by a smooth rational curve or a curve of arithmetic genus one (the generic case)).
Elliptic fibrations are encoded in the lattice $\Num(Y)$ in terms of isotropic vectors.
Indeed, pick a primitive isotropic vector $E\in\Num(Y)$
which we may assume to be effective after changing sign, if necessary, by Riemann--Roch.
Then the linear system $|2E|$ may still contain some base locus (which consists of $(-2)$-curves),
but after subtracting the base locus, the resulting isotropic vector will induce an elliptic fibration.

\begin{Proposition}
\label{prop:isotropic}
Let $Y$ be an Enriques surface supporting a root type $R$ of rank $9$.
Then there is a primitive isotropic vector $E\in\Num(Y)$ such that
\begin{enumerate}
\item
there is a smooth rational curve $C$ on $Y$ with $C.E=1$,
\item
$R_0=R\cap E^\perp\subset\Num(Y)$ is a root lattice of rank $8$.
\end{enumerate}
\end{Proposition}

%\begin{proof}
The proof of Proposition \ref{prop:isotropic} 
%largely build on the gluing isomorphism \eqref{eq:q};
%indeed for every root lattice of rank 9 embedding into $U+E_8$,
%this will provide us more or less directly with an isotropic vector
%such that both conditions are fulfilled.
amounts to a case-by-case analysis
for all the 38  root types $R$ from \ref{ss:31} and \ref{ss:add}.
It purely bases on the gluing data analyzed in Section \ref{s:latt}.
More precisely, the isotropic vector $E$ will be given using the positive class $h$ and some vector $e^\vee\in (R')^\vee$
which we choose as dual vector of a single vertex $e\in R$.
Once this  is realized, the second claim of Proposition \ref{prop:isotropic} follows readily
(with $R_0 = R\setminus\{e\}$).
For the first, one may take precisely the smooth rational curve on $Y$
corresponding to the  vertex $e\in R$;
in particular, $e.E=1$, so $E$ is primitive as stated.
For the details, we slightly modify a table in the suggestive notation from \cite{S-Q-hom} (to be recalled below)
and supplement it by the corresponding data for the 7 additional root types from \ref{ss:add}.
%with two entries added at the very bottom for the potential additional types in characteristic $3$.
%\end{proof}

\subsection{Notation for dual vectors}

The gluing argument from the proof of Proposition \ref{prop:isotropic}
uses dual vectors of single vertices of $ADE$ root lattices.
Below we sketch the notation to be used in Table \ref{T1}.
For details we refer to \cite[\S 2]{S-Q-hom}.

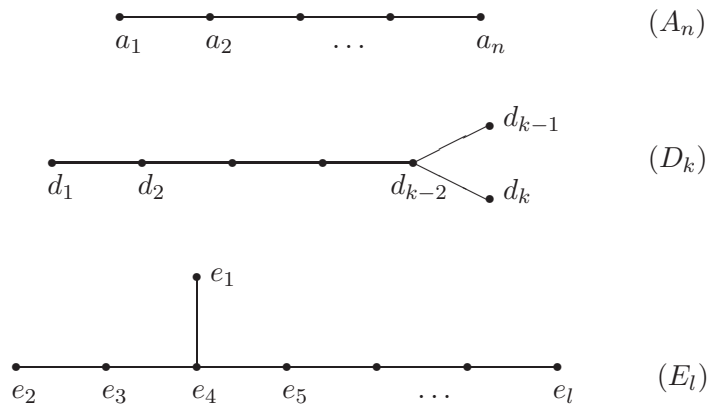
\begin{figure}[ht!]
\setlength{\unitlength}{.6mm}
\begin{picture}(80,22)(0,17)
%\put(-15,31){$A_n$}
\multiput(3,32)(20,0){5}{\circle*{1.5}}
\put(3,32){\line(1,0){80}}
\put(2,25){$a_1$}
\put(22,25){$a_2$}
\put(50,25){$\hdots$}
\put(82,25){$a_n$}
\put(120,29){$(A_n)$}

%
%\put(-15,7){$D_r:$}
%\multiput(3,8)(20,0){5}{\circle*{3}}
%\put(3,8){\line(1,0){80}}
%\put(83,8){\line(2,1){17}}
%\put(83,8){\line(2,-1){17}}
%\put(100,16){\circle*{3}}
%\put(100, 0){\circle*{3}}
%\put(2,1){$\alpha_1$}
%\put(22,1){$\alpha_2$}
%\put(80,1){$\alpha_{r-2}$}
%\put(105,16){$\alpha_{r-1}$}
%\put(105, 0){$\alpha_{r}$}
%
%\put(-15,-33){$E_r:$}
%\multiput(3,-32)(20,0){6}{\circle*{3}}
%\put(3,-32){\line(1,0){100}}
%\put(2,-39){$\alpha_1$}
%\put(22,-39){$\alpha_2$}
%\put(42,-39){$\alpha_3$}
%\put(62,-39){$\alpha_4$}
%\put(43,-32){\line(0,1){20}}
%\put(43,-12){\circle*{3}}
%\put(47,-13){$\alpha_r$}
%\put(102,-39){$\alpha_{r-1}$}

\end{picture}
%\caption{Dynkin diagram of type $A_n$}  
%\label{DynkinA}
\end{figure}

\begin{figure}[ht!]
\setlength{\unitlength}{.6mm}
\begin{picture}(100,20)(5,-5)
%\put(-15,31){$A_r:$}
%\multiput(3,32)(20,0){5}{\circle*{3}}
%\put(3,32){\line(1,0){80}}
%\put(2,25){$\alpha_1$}
%\put(22,25){$\alpha_2$}
%\put(82,25){$\alpha_r$}
%
%\put(-15,7){$D_k:$}
\multiput(3,8)(20,0){5}{\circle*{1.5}}
\put(3,8){\line(1,0){80}}
\put(83,8){\line(2,1){17}}
\put(83,8){\line(2,-1){17}}
\put(100,16){\circle*{1.5}}
\put(100, 0){\circle*{1.5}}
\put(2,1){$d_1$}
\put(22,1){$d_2$}
\put(78,1){$d_{k-2}$}
\put(103,16){$d_{k-1}$}
\put(103, 0){$d_{k}$}
\put(135,7){$(D_k)$}

%\put(-15,-33){$E_r:$}
%\multiput(3,-32)(20,0){6}{\circle*{3}}
%\put(3,-32){\line(1,0){100}}
%\put(2,-39){$\alpha_1$}
%\put(22,-39){$\alpha_2$}
%\put(42,-39){$\alpha_3$}
%\put(62,-39){$\alpha_4$}
%\put(43,-32){\line(0,1){20}}
%\put(43,-12){\circle*{3}}
%\put(47,-13){$\alpha_r$}
%\put(102,-39){$\alpha_{r-1}$}

\end{picture}
%\caption{Dynkin diagrams of type $D_k$}  
%\label{DynkinAD}
\end{figure}

\begin{figure}[ht!]
\setlength{\unitlength}{.6mm}
\begin{picture}(120,30)(3,-40)
%\put(-15,31){$A_r:$}
%\multiput(3,32)(20,0){5}{\circle*{3}}
%\put(3,32){\line(1,0){80}}
%\put(2,25){$\alpha_1$}
%\put(22,25){$\alpha_2$}
%\put(82,25){$\alpha_r$}
%%
%\put(-15,7){$D_r:$}
%\multiput(3,8)(20,0){5}{\circle*{3}}
%\put(3,8){\line(1,0){80}}
%\put(83,8){\line(2,1){17}}
%\put(83,8){\line(2,-1){17}}
%\put(100,16){\circle*{3}}
%\put(100, 0){\circle*{3}}
%\put(2,1){$\alpha_1$}
%\put(22,1){$\alpha_2$}
%\put(80,1){$\alpha_{r-2}$}
%\put(105,16){$\alpha_{r-1}$}
%\put(105, 0){$\alpha_{r}$}
%
%\put(-15,-33){$E_r:$}
\multiput(3,-32)(20,0){7}{\circle*{1.5}}
\put(3,-32){\line(1,0){120}}
\put(2,-39){$e_2$}
\put(22,-39){$e_3$}
\put(42,-39){$e_4$}
\put(62,-39){$e_5$}
\put(43,-32){\line(0,1){20}}
\put(43,-12){\circle*{1.5}}
\put(46,-13){$e_1$}
\put(92,-39){$\hdots$}
\put(122,-39){$e_l$}
\put(145,-36){$(E_l)$}

\end{picture}
%\caption{Dynkin diagrams of type $A_r, D_r, E_r$}  
%\label{DynkinAD}
\end{figure}

\begin{Remark}
We note that the choices made in the proof of Proposition \ref{prop:isotropic}
are far from being unique as $Y$ may contain many isotropic vectors, and many elliptic fibrations.
\end{Remark}

\begin{table}
$$
\begin{array}{|ccccc|}
\hline
R & R' & |d| & E & R_0\\
\hline
\hline
A_9 & A_9 & 10 & a_2^\vee + 2h/5 & A_7+A_1\\
%A_8+A_1 & A_8+A_1 & 18 & (a_1^\vee,0)+2h/9 & A_7+A_1\\
%& E_8+A_1 & 2 & (0,a_1^\vee) + h/2 & A_8\\
A_8+A_1 & A_8+A_1 & 18 & (a_3^\vee,0)+h/3 & A_5+A_2+A_1\\
& E_8+A_1 & 2 & (a_3^\vee,0) + h & A_5+A_2+A_1\\
A_7 + A_2 & E_7 + A_2 & 6 & (a_2^\vee,0) + h/2 & A_5+A_2+A_1\\
A_7 + 2A_1 & E_8 + A_1 & 2 & (a_4^\vee,0,0) + h & 2A_3+2A_1\\
A_6+A_2+A_1 & A_6+A_2+A_1 & 42 & (a_1^\vee,0,0)+h/7 & A_5+A_2+A_1\\
A_5 + A_4 & A_5+A_4 & 30 & (0,a_2^\vee)+h/5 & A_5 + A_2 + A_1\\
A_5+A_3+A_1 & E_6+A_3 & 24 & (a_2^\vee,0,0)+h/3 & 2A_3+2A_1\\
A_5+2A_2 & E_7+A_2 &
6 & (0,0,a_1^\vee)+ h/3 & A_5+A_2+A_1\\
A_5+A_2+2A_1 & E_8+A_1 &
2 & (0,0,0,a_1^\vee)+h/2 & A_5+A_2+A_1\\
2A_4+A_1 & E_8 + A_1 &
2 &  (0,0,a_1^\vee) + h/2 & 2A_4\\
A_4+A_3+2A_1 & A_4+D_5 &
20 & (a_1^\vee,0,0,0)+h/5 & 2A_3+2A_1\\
3A_3 & D_9 &
4 & (a_2^\vee,0,0) + h/2 & 2A_3+2A_1\\
2A_3+A_2+A_1 & E_7 + A_2 &
6 & (0,0,a_1^\vee,0) + h/3 & 2A_3+2A_1\\
2A_3+3A_1 & E_8+A_1 &
2 & (0,0,0,0,a_1^\vee)+h/2 & 2A_3+2A_1\\
A_3 + 3A_2 & A_3 + E_6 &
 12 & (a_1^\vee,0,0,0) + h/4 & 4A_2 \\
D_9 & D_9 &
4 & d_9^\vee + 3h/4 & A_8\\ 
D_8+A_1 & E_8+A_1 &
2 & d_8^\vee + h & A_7+A_1\\
%D_8 + A_1 & E_8 + A_1 & (0,a_1^\vee)+h/2 & D_8\\
D_7+2A_1 & D_9 &
4 & (d_1^\vee,0,0)+h/2 & D_6+2A_1\\
D_6+A_3 & D_9 &
4 & (0,a_2^\vee)+h/2 & D_6+2A_1\\
D_6+A_2+A_1 & E_7 + A_2 & 
6 & (0,a_1^\vee,0) + h/3 & D_6+2A_1\\
%D_6+3A_1 & E_8+A_1 &
%2 & (0,0,0,a_1^\vee) + h/2 & D_6+2A_1\\
D_6+3A_1 & E_8+A_1 &
2 & (d_2^\vee,0,0,0) + h & D_4+4A_1\\
D_5+A_4 & D_5+A_4 &
20 & (d_5^\vee,0) + h/4 & 2A_4\\
%D_5+A_3+A_1 & E_8+A_1 &
%2 & (0,0,a_1^\vee) + h/2 & D_5 + A_3\\
D_5+A_3+A_1 & E_8+A_1 &
2 & (d_2^\vee,0,0) + h & 2A_3+2A_1\\
D_5+D_4 & D_9 &
4 & (0,d_1^\vee)+h/2 & D_5+A_3 \\
D_4+A_3+2A_1 & D_9 &
4 & (d_1^\vee,0,0,0)+h/2 & 2A_3+2A_1\\
2D_4+A_1 & E_8+A_1 &
2 & (0,0,a_1^\vee)+h/2 & 2D_4\\
% old
%E_8+A_1 & E_8+A_1 &
%2 & (0,a_1^\vee) + h/2 & E_8\\
%E_7+A_2 & E_7+A_2 &
%6 & (e_7^\vee,0) + h/2 & E_6+A_2\\
%E_7 + 2A_1 & E_8 + A_1 &
%2 & (0,0,a_1^\vee) + h/2 & E_7 + A_1 \\
E_8+A_1 & E_8+A_1 &
2 & (e_8^\vee,0) + h & E_7+A_1\\
E_7+A_2 & E_7+A_2 &
6 & (e_3^\vee,0) + h & A_5+A_2+A_1\\
%6 & (e_7^\vee,0) + h/2 & E_6+A_2\\
E_7 + 2A_1 & E_8 + A_1 &
2 & (e_2^\vee,0,0) + h & D_6 + 2A_1 \\
E_6+A_3 & E_6+A_3 &
12 & (e_6^\vee,0)+h/3 & D_5+A_3\\
%E_6+A_2+A_1 & E_8+A_1 &
%2 & (0,0,a_1^\vee)+h/2 & E_6+A_2\\
E_6+A_2+A_1 & E_8+A_1 &
2 & (e_1^\vee,0,0)+h & A_5+A_2+A_1\\
\hline
\hline
D_5+4A_1 & D_9 & 4 & d_1^\vee + h/2 & D_4+4A_1\\
D_4+5A_1 & E_8+A_1 & 2 & a_1^\vee + h/2 & D_4+4A_1\\
D_4+A_2+3A_1 & E_7+A_2 & 6 & a_2^\vee+h/3 & D_4+4A_1\\
A_3+6A_1 & D_9 & 4 & a_2^\vee + h/2 & 8A_1\\
4A_2+A_1 & E_8+A_1 &
2 & (0,0,0,0,a_1^\vee)+h/2 & 4A_2\\
A_2+7A_1 & A_2+E_7 &
6 & a_2^\vee,+h/3 & 8A_1\\
9A_1 & E_8+A_1 & 2 & a_1^\vee + h/2 & 8A_1\\
\hline
\end{array}
$$
\caption{Isotropic vectors and root lattices for the proof of Proposition \ref{prop:isotropic}}
\label{T1}
\end{table}

\begin{Convention}
Given a primitive isotropic vector $E$ as above,
we will always refer to the associated elliptic fibration 
as associated to the linear system $|2E|$,
regardless of the subtleties of sign and base locus.
\end{Convention}

\subsection{Jacobian elliptic surface}
\label{ss:jac}

To make the whole situation explicit,
let us consider the Jacobian of \eqref{eq:ell}.
This is a rational elliptic surface $\Jac(f)$
from which $Y$ can be recovered as a torsor
(or over $\C$, by a logarithmic transformation).
Here this surface turns out to be rather special;
namely it will be extremal:

\begin{Definition}
\label{def}
A rational elliptic surface $S\to \PP^1$ is called \emph{extremal}
if its Mordell--Weil group $\MW(S)$ is finite.
\end{Definition}

Explicit examples will be given in Section \ref{s:res}.
For now we content ourselves with the following important consequence
of our findings so far:

\begin{Corollary}
\label{cor:ell}
Let $Y$ be an Enriques surface supporting a root type $R$ of rank $9$.
Then $Y$ admits an elliptic fibration % \eqref{eq:ell}
with a smooth rational bisection
such that $\Jac(Y)$ is an extremal rational elliptic surface.
\end{Corollary}

\begin{proof}
All we need is already contained in Proposition \ref{prop:isotropic}
whose notation we will employ.
For instance, if $|2E|$ already gives an elliptic fibration,
then $C$ is a bisection,
and the root lattice $R_0$ is supported on fiber components.
Since the same singular fibers feature on $\Jac(Y)$,
the claim follows from the Shioda--Tate formula \cite{ShEMS}.

If $|\pm 2E|$ involves some base locus,
then subtracting this amounts
to Picard--Lefschetz reflections in the corresponding smooth rational curves.
In particular, this may affect the $(-2)$-curves $C_1,\hdots,C_9$ representing $R$,
but in essence this may only add smooth rational curves.
Indeed, any $C_i$ is mapped to some effective or anti-effective $(-2)$-divisor
which itself is supported on smooth rational curves.
Thus we do obtain a smooth rational bisection (potentially different from $C$)
and a root lattice $R_0$ supported on fiber components as before.
\end{proof}

\section{Ruling out the additional root types}

Many readers may be used to the situation where a root lattice $R_0$ supported on the singular fibers
of an elliptic fibration determines the fibers themselves 
(up to some ambiguity within Kodaira types $I_2/III$ and $I_3/IV$)
as the extended Dynkin diagrams $\tilde R_v$ corresponding to the orthogonal summands of $R_0$;
conversely, the summands can be recovered by omitting any simple component from each reducible fiber involved.
Presently, this will indeed often be the case, but not always
(leading to different moduli components, in fact).
We shall make this more precise in the following very useful criterion.
Then we shall
illustrate the method with an example in \ref{ex2}
 before turning to the additional root types from \ref{ss:add} in \ref{ss:no-add}.

 \subsection{Criterion for fiber configurations}
 
 By inspection, for all data in Table \ref{T1},
 the isotropic vector $E$ involves a fraction of the positive vector $h$
 and the dual vector $e^\vee$ of a single curve $e\in R$.
 In particular, $R_0=R\setminus\{e\}$.
 This set-up often already suffices to rule out many fiber configurations for $|2E|$
 as we shall record in the following criterion.
 
 \begin{Criterion}
 \label{crit-e}
 If  less than  two  of the curves forming $R_0$ meet $e$,
 then the only possible fibre configuration for $|2E|$ is comprised of  fibers of type $\tilde R_v$
 for the orthogonal summands $R_v$ of $R_0$. 
 \end{Criterion}
 
 \begin{Remark}
 Note that there still may be multiplicities at the multiplicative fibers involved.
 \end{Remark}
 
 \begin{proof}
 Let $\tilde\mR$ denote an extended Dynkin diagram with underlying irreducible root lattice $\mR$.
 Assume that the orthogonal summands $R_1,\hdots,R_m \; (m\geq 1)$ embed into a singular fiber of type $\tilde\mR$
 by way of the reflection $\sigma$.
 By Corollary \ref{cor:ell}, we may assume that the $R_i$ have the same total rank as $\mR$;
 in particular, their support avoids at most one component of $\tilde\mR$.
 Now there are two cases to distinguish.
 
If all $R_i$ are orthogonal to $e$, then
 the bisection $B=\sigma(e)$ meets a single simple fiber component $\Theta$ of $\tilde R$ which is avoided by all $R_i$,
i.e.~the $R_i$ embed as finite index sublattice into $\tilde\mR\setminus\{\Theta\}=\mR$ via $\sigma$.
But then 
\begin{eqnarray*}
\label{eq:mR}
\bigoplus_{i=1}^m \sigma(R_i) \hookrightarrow \mR
\end{eqnarray*}
is primitive by Lemma \ref{lem:no};
on the other hand, $\mR$ is irreducible by assumption, so $m=1$ and $\mR = R_1$ as claimed.

In the other case, if a single curve $e'$ in the $R_i$ meets $e$, then $e'.e=1$
implies that 
%\begin{itemize}
%\item
either the fiber is multiple and thus multiplicative,
so the rank condition only allows for $\tilde \mR = \tilde R_1$ (i.e.~$m=1$);
%\item
or 
$B$ meets two different fiber components $\Theta, \Theta'$ 
of which exactly $\Theta'$, say, is contained in the support of $\sigma(e')$
while both components are off the images of the other curves in the $R_i$ under $\sigma$.
That is, the $R_i$ embed already into $\tilde\mR\setminus\{\Theta\}=\mR$,
and we conclude as before.
%\end{itemize}
 \end{proof}

\begin{Corollary}
\label{cor:noe8}
In the above framework,
there are never two orthogonal root lattices embedding into $\tilde E_8$.
\end{Corollary}

 \begin{Remark}
 \label{rem:crit}
 It is obvious from the proof of Criterion \label{crit-e}
 that the statement holds fiberwise.
 That is, if some fiber contains more than one orthogonal summand of $R_0$,
 then at least two curves from these summands have to meet $e$.
 \end{Remark}

\subsection{Prototype example ($R=A_5+2A_2$ cont'd)}
\label{ex2}

With the primitive isotropic vector $E$ from Table \ref{T1},
there are several alternatives for singular fibers supporting the root lattice $R_0=A_5+A_2+A_1$:
\[
\tilde A_5+\tilde A_2+\tilde A_1, \; \tilde E_6+\tilde A_2,\; \tilde E_7+\tilde A_1, \tilde E_8.
\]
However, by Criterion \ref{crit-e},
only the first alternative may be compatible with the extra information
provided by the smooth rational bisection $B$ induced by the curve $a_2$.
%
%
% (after a suitable composition of reflections,
%cf.~the discussion in \cite[\S3]{S-Q-hom}).
To check the compatibility, we can simply draw an appropriate diagram of $(-2)$-curves
comprising fiber components and bisection
(where the dashed arrows mean that the fiber may be multiple (though not simultaneously 
when involving the $I_3$ fiber)).

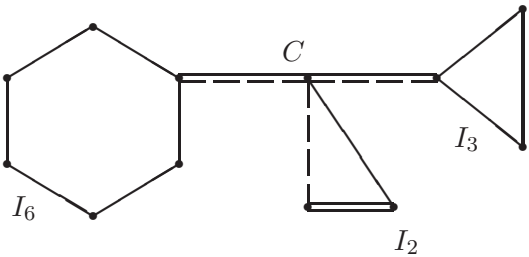
\begin{figure}[ht!]
\setlength{\unitlength}{.45in}
\begin{picture}(8,2.9)(3,0)
\thicklines

%I_6

\multiput(4,2)(2,0){2}{\circle*{.1}}
\multiput(4,1)(2,0){2}{\circle*{.1}}
\put(4,2){\line(0,-1){1}}
\put(6,2){\line(0,-1){1}}
\put(5,0.4){\circle*{.1}}
\put(5,2.6){\circle*{.1}}

%\put(5,0.4){\circle{0.2}}

\put(4,2){\line(5,3){1}}
\put(5,0.4){\line(5,3){1}}
\put(4,1){\line(5,-3){1}}
\put(5,2.6){\line(5,-3){1}}

% sections
%\put(3,1.4){\circle{.1}}
%\put(3,2){\circle{.1}}
%\put(3,2.6){\circle{.1}}
%\put(3,2){\line(1,0){1}}
%
%\put(3,2.6){\line(5,-3){1}}
%\put(6,2){\line(5,3){1}}

%\put(7,1.4){\circle{.1}}
\put(7.5,2){\circle*{.1}}
%\qbezier(5,2.6)(6.5,2.6)(7.5,2)
%\qbezier(5,.4)(6.5,.4)(7.5,2)

\put(7.5,2.04){\line(-1,0){1.5}}
\multiput(7.5,1.96)(-.4,0){4}{\line(-1,0){.3}}

%I_3

\put(9,2){\circle*{.1}}
%\put(7.5,1.96){\line(1,0){1.5}}
\multiput(7.5,1.96)(.4,0){4}{\line(1,0){.3}}
\put(7.5,2.04){\line(1,0){1.5}}

\put(10,2.8){\circle*{.1}}
\put(10,1.2){\circle*{.1}}

\put(9,2){\line(5,4){1}}
\put(9,2){\line(5,-4){1}}

\put(10,2.8){\line(0,-1){1.6}}

%I_2

\put(7.5,.5){\circle*{.1}}
\put(8.5,.5){\circle*{.1}}

\put(7.5,.46){\line(1,0){1}}
\put(7.5,.54){\line(1,0){1}}

%\put(7.5,2){\line(0,-1){1.5}}
\multiput(7.5,1.96)(0,-.4){4}{\line(0,-1){.3}}
\put(7.5,2){\line(2,-3){1}}

%\put(7,2.6){\circle{.1}}
%\put(6,2){\line(1,0){1}}
%
% letters
\put(7.2,2.2){$C$}
\put(8.5,0){$I_2$}
\put(9.2,1.2){$I_3$}
\put(4.05,.4){$I_6$}
%\put(1.8,1.6){$x^3=1$}
%\put(4.1,1.8){$C_y$}
%\put(5.55,1.8){$C_z$}
%\put(4.1,1){$E$}
%
%%
\thinlines
%\put(3.8,0.8){\framebox(.4,1.4){}}
%\put(5.8,0.8){\framebox(.4,1.4){}}
%\put(9.8,1){\framebox(.4,1.9){}}
%
%\put(7.3,0.3){\framebox(.4,1.9){}}
%
\end{picture}
\caption{Configurations supporting the root type $R=A_5+2A_2$}
\label{Fig:R}
\end{figure}

%
%To see that the other cases cannot occur,
%we argue with the composition of reflections $\sigma$
%which takes $E$ to a half-pencil and $a_2$ to the bisection $B$.
%Then all curves from $R_0$ are mapped to effective or anti-effective $(-2)$-divisors supported on the singular fibers
%all of which except for $\sigma(a_2)$ are perpendicular to $B$.
%This rules out $\tilde E_8$ right away since there $B$ meets a single fiber component with multiplicity $2$,
%so $\sigma(a_2).B=1$ would not be possible.
%Similarly, if 
%\[
%A_5+A_2\hookrightarrow\tilde E_7,
%\]
%then $B$ would again meet a single fiber component
%(off both root lattices), so in fact
%\[
%A_5+A_2\hookrightarrow E_7.
%\]
%But then there would be an effective $3$-divisibility, contradicting Lemma \ref{lem:no}.
%Along the same lines, if 
%\[
%A_5+A_1\hookrightarrow\tilde E_6,
%\]
%then $B$ would meet two different fiber components (one off $\sigma(A_1)$, both off $\sigma(A_5)$),
%so as above 
%\[
%A_5+A_1\hookrightarrow E_6
%\]
%gives an effective $2$-divisibility and thus the desired contradiction.
%

\subsection{Additional root types}
\label{ss:no-add}

We now turn to the additional root types from \ref{ss:add}
which we could not rule out in positive characteristic
because we cannot apply the orbifold Bogomolow--Miyaoka--Yau inequality.
Yet we claim that they do not occur:

\begin{Proposition}
\label{lem:no-add}
No root type from \ref{ss:add} is supported on an Enriques surface in odd characteristic.
\end{Proposition}

\begin{proof}
We first assume that $R\neq 4A_2+A_1$ (so in particular $R$ is $2$-elementary).
By Corollary \ref{cor:ell}, the data from Table \ref{T1} endows $Y$ with an elliptic fibration
whose jacobian is an extremal rational elliptic surface.
More precisely, except for $R=A_3+6A_1$, 
Criterion \ref{crit-e} predicts that
the singular fibers are given as $\tilde R_v$ for the orthogonal summands of $R_0$.
However, this would cause the Euler-Poincar\'e characteristic to exceed $12$,
contradiction.

For $R=A_3+6A_1$, the lattice $R_0=8A_1$ could a priori embed into a number of fiber configurations:
\[
8\tilde A_1, \;\; \tilde D_4+4\tilde A_1,  \;\; 2\tilde D_4, \;\; \tilde D_6 + 2\tilde A_1, \;\; \tilde E_7+\tilde A_1,\;\;
\tilde E_8.
\]
The first two, however, are ruled out by Euler-Poincar\'e characteristic considerations as above, 
the third by Remark \ref{rem:crit} and 
the last by Corollary \ref{cor:noe8}.
So assume that 
\[
R_0 \stackrel{\sigma}{\hookrightarrow} \tilde D_6+2\tilde A_1.
\]
Write $a_1, a_3\in A_3\cap R_0$ and $a_4,\hdots, a_7$ for the generators of the other four copies of $A_1$
embedding into $\tilde D_6$.
Clearly the bisection $B$ meets two different fiber components $\Theta, \Theta'$.
If the support of $\sigma(a_1), \sigma(a_3)$ were to contain only one and the same component, say $\Theta$,
then 
\[
R_0\hookrightarrow \tilde D_6\setminus\{\Theta'\} = D_6,
\]
giving a contradiction by Lemma \ref{lem:no} (as in the proof of Criterion \ref{crit-e}).
Otherwise, we can apply reflections in fiber components to ensure that
\[
\sigma(a_1) = \Theta, \;\;\; \sigma(a_3) = \Theta'.
\]
It follows that
the support of the remaining $(-2)$-divisors $\sigma(a_4),\hdots,\sigma(a_7)$
contains neither $\Theta$ or $\Theta'$ nor the components adjacent to these two
(since all $a_i$ are orthogonal).
Depending on the precise location of $\Theta$ and $\Theta'$ in the fiber,
this implies
\[
4A_1 \hookrightarrow D_4 \;\;\; \text{ or } \;\;\; 4A_1 \hookrightarrow 3A_1,
\]
each of which is impossible by Lemma \ref{lem:no} or for rank reasons.
Similarly, if  
\[
R_0 \stackrel{\sigma}{\hookrightarrow} \tilde E_7+\tilde A_1,
\]
then the analogous reasoning leads to a contradiction by way of the embeddings
\[
7A_1 \hookrightarrow E_7 \;\;\; \text{ or } \;\;\; 5A_1 \hookrightarrow D_4.
\]
This excludes the root type $R=A_3+6A_1$.

It remains to discuss the root type $R=4A_2+A_1$
with bisection $B$ off $R_0$ induced by the data in Table \ref{T1}.
By Criterion \ref{crit-e}, the singular fibers are determined as $4\tilde A_2$,
%The same argument as in \ref{ex2} rules out
%\[
%4A_2\hookrightarrow \tilde E_6+\tilde A_2, \tilde E_8,
%\]
%so we are left with the  case
%\[
%4A_2\hookrightarrow 4\tilde A_2,
%\]
i.e.~$Y$ (and thus $\Jac(Y)$) has four fibers of Kodaira type $I_3$.
This may seem plausible, yet we claim that it does not occur.
In characteristic $3$, this follows readily from the classification
of  semi-stable extremal rational elliptic surfaces in \cite{Lang1}
(or from the degeneration of the group scheme of $3$-torsion points in characteristic $3$,
as this causes the Hesse pencil $X_{3333}$ to degenerate, cf.~Table \ref{T0}).
Otherwise, we have $p>3$ and consider the K3 cover $X$ which is endowed with an elliptic fibration
with 8 fibers of type $I_3$ (or potentially some merging to type $I_6$),
and disjoint sections $O, P$ which $B$ splits into.
Since $O$ and $P$ are orthogonal, the height of $P$
in the language of Mordell--Weil lattices \cite{ShMW}
only depends on the contraction terms at the fibers met non-trivially,
i.e.~in a different component than $O$:
\[
h(P) = 4 - \sum_v \mbox{contr}_v(P).
\]
Since on $Y$ the bisection $B$ meets every singular fiber in a single component 
(since $A_1$ inducing $B$ is perpendicular to the four copies of $A_2$),
the same holds true for its pre image $O+P$ on $X$
unless the fiber in question is ramified (attaining type $I_6$).
In the latter case, $O$ and $P$ meet exactly opposite fiber components, so
\[
h(P) =
\begin{cases}
4 & \text{ if there is no ramified fiber,}\\
4 - \frac 32 = \frac 52 & \text{ if there is one ramified fiber,}\\
4 - 2\cdot\frac 32 = 1 & \text{ if there are two ramified fibers.}\\
\end{cases}
\]
It follows that the given fiber components and sections generate a sublattice 
\begin{eqnarray}
\label{eq:emb}
L \hookrightarrow \NS(X)
\end{eqnarray}
of rank $19$ (resp.~$20, 21$).
We continue by comparing the $3$-lengths of these lattices,
i.e.~the minimum number of generators of the $3$-torsion subgroup of the discriminant group.
Presently $L$ has  $3$-length $8$ (resp.~$7, 6$),
but the embedding \eqref{eq:emb} need not be primitive
because there may be three-torsion sections involved.
In any case the primitive closure $L'$ will have $3$-length at least $4$ (resp.~$3,2$).
To see this, we use that full $3$-torsion in $\MW(X)$ may let the $3$-length drop by $4$ at most,
but on top of that neither the section $P$ nor any of its translates by torsion sections may be $3$-divisible,
since otherwise the resulting section would have height $4/9$ (resp.~$5/18$ or $1/9$), but 
the height pairing has image in $\frac 16\Z$ by inspection of the present singular fibers
and their correction terms (which lie in the set $\{0,2/3,5/6,4/3,3/2\}$ by \cite{ShMW}).

On the other hand,  since the Tate conjecture is known for elliptic K3 surfaces by \cite{ASD},
either $X$ is supersingular and $\NS(X)$ is $p$-elementary of rank $22$,
or by \cite{Deligne}, $X$ lifts to characteristic zero with all of $\NS(X)$ which thus embeds into the K3 lattice $\Lambda=3U+2E_8$.
Either case imposes the rank inequality
\[
(\text{$3$-length})(L') = (\text{$3$-length})(L^\perp) \leq \text{rank}(L^\perp) = 22-\mbox{rank}(L')
\]
where the orthogonal complement is taken in $\NS(X)$ in the supersingular case resp.~in $\Lambda$.
Since this inequality never holds  presently, we conclude that
the root type $R=4A_2+A_1$ cannot be supported on an Enriques surface.
This completes the proof of Proposition \ref{lem:no-add}.
\end{proof}

\begin{Corollary}
There are no other root types supported on Enriques surfaces in odd characteristic compared to characteristic zero.
\end{Corollary}

\subsection{Orbifold Bogomolow--Miyaoka--Yau inequality}

As a byproduct of Proposition \ref{lem:no-add} we record the following consequence:

\begin{Corollary}
The orbifold Bogomolow--Miyaoka--Yau inequality holds for any $\Q_\ell$-cohomology projective plane
obtained from an Enriques surface.
\end{Corollary}

\begin{Remark}
Of course, the above argument works in characteristic zero, too.
Hence, in the classification of maximal root types supported on Enriques surfaces in \cite{HKO} and \cite{S-Q-hom},
there is no formal need to appeal to the orbifold Bogomolov--Miyaoka--Yau inequality.
\end{Remark}

\section{Base change construction}
\label{s:bc}

Usually Enriques surfaces are not so easy to get a hand on explicitly,
especially with regards to the smooth rational curves which we presently care about.
Here one may equally well think about Enriques' original model, Horikawa's model or the torsor interpretation
alluded to in \ref{ss:jac}.
Presently, however, the root lattice $R$ of rank $9$ endows $Y$ with so much structure
that everything can be made totally explicit.
This is due to the special elliptic fibration from Corollary \ref{cor:ell}.
Notably, $\pi^*E$ induces an elliptic fibration on the K3 cover
\[
\pi^*f: \; X \longrightarrow \PP^1
%X\stackrel{|\pi^*E|}{\longrightarrow}\PP^1
\]
which generically would have no section,
but presently clearly does as the smooth rational bisection $B$ of $|2E|$ splits into two disjoint smooth rational curves on $X$
which are exchanged by the Enriques involution $\tau$ (as exploited in the last part of \ref{ss:no-add}).
We denote them by $O$ and $P$ where the former indicates the chosen zero section of $\pi^*f$.

It is exactly this situation where a base change construction of Kond\=o \cite{Kondo-Aut} kicks in
(later extended in \cite{HS}).
Since Kond\=o's original argument works only over $\C$ as it builds on Nikulin's classification 
of non-symplectic automorphisms of complex K3 surfaces,
we shall give a proof which also works in odd characteristic.
(A similar argument is given independently in work of  Martin \cite{Martin}.)

\begin{Proposition}
\label{prop:imath}
There is an involution $\imath$ on $X$
such that 
\[
\tau = (\text{translation by } P) \circ \imath.
\]
\end{Proposition}

\begin{proof}
Let $t_P$ denote translation by $P$, defined fiberwise and extended to a symplectic automorphism of $X$.
Set
\[
\imath = t_{-P}\circ\tau = t_P^{-1}\circ\tau.
\]
Then $\imath$ fixes $O$ as a set, with two fixed points in the ramified fibers.
We claim that $\imath^2=1$.
To see this, note that $\imath^2$ fixes
%\begin{itemize}
%\item
fibers as sets and
%\item
$O$ pointwise.
%\end{itemize}
Hence $\imath^2$ acts as an automorphism on the generic fiber $X_\eta$ of $\pi^*f$.
In particular,
\[
\mbox{ord}(\imath^2)\in\{1,2,3,4,6\}.
\]
Outside characteristic $3$, the order corresponds to the action on the regular $1$-form of $X_\eta$,
i.e.~$\imath^2$ has order $n$ if and only if it acts by a primitive $n$-th root of unity on $H^0(X_\eta, \omega_{X_\eta})$.
Since given a minimal Weierstrass form, we locally have $\omega_X=dt\wedge\omega_{X_\eta}$,
the same holds for its action on $H^0(X,\omega_X)$.
But by definition, $\imath^2$ is symplectic, so $\imath^2=1$.

In characteristic $3$, this argument may only break down for a wild automorphism,
more precisely~if $\mbox{ord}(\imath^2)=3$ (order $6$ is ruled out by the above reasoning because $\imath^2$ is symplectic.)
This in turn implies that $X_\eta$ has zero j-invariant.
In particular, $\pi^*f$ is an isotrivial fibration,
and all singular fibers are additive.
But then the ramified fibers necessarily are smooth (since otherwise they would have to be multiplicative),
and $\imath^2$ fixes just one point on them.
In detail, if the generic fiber is given by 
\[
y^2 = x^3 - \alpha^2 x, \;\;\; \alpha\in K(t),
\]
then $\imath^2$ would have to act by 
\[
(x,y)\mapsto (x+\alpha,y) \;\; \text{ or } \;\; (x-\alpha,y),
\]
and this action directly extends to the all smooth fibers including $O$ being the single fixed point.
On the other hand, $\imath$ itself acts as an automorphism of a ramified fiber $X_0$
as it fixes $O|_{X_0}$. In particular, from the definition of $\imath$ we can infer 
\[
\imath(P|_{X_0}) = -P|_{X_0} \; \Longrightarrow \; \imath^2(P|_{X_0}) = P|_{X_0},
\]
i.e. $\imath^2$ has at least two fixed points on the ramified fiber $X_0$, contradiction.
\end{proof}

With the involution $\imath$ at our disposal,
the reasoning from \cite{Kondo-Aut} applies to prove that
$X$ is in fact a quadratic base change from the rational elliptic surface $\Jac(f)$
(hence the title of this section).

\begin{Corollary}
$X/\imath=\Jac(f)$.
\end{Corollary}

\begin{Remark}
The base change type construction may also be applied to the problem
of Enriques surfaces with finite automorphism group in positive characteristic
(much like what Kond\=o did over $\C$, but now without reference to Torelli type theorems).
For details, see upcoming work of Martin \cite{Martin}.
\end{Remark}

\section{Translation into equations}
\label{s:eqns}

It is about time to collect what we have found so far for Enriques surfaces  $Y$ supporting a root type $R$ from \eqref{eq:list}, 
and start to take advantage of it
following \cite{HS}.
To this end, we note two crucial observations:

\begin{Observation}
\label{obs}
\begin{enumerate}
\item\label{it1}
the K3 cover $X$ is uniquely determined by the rational elliptic surface $\Jac(f)$
and the base change which in turn only depends on  the location of the ramified fibers;
both information we can easily read off from $Y$;
\item\label{it2}
the section $P$ which is used to define the Enriques involution $\tau=t_P\circ \imath$
is anti-invariant for $\imath$, i.e.~$\imath^*(P) = -P$ in $\MW(X)$.
\end{enumerate}
\end{Observation}

In consequence, considering \eqref{it2}, the section $P$ is invariant 
for the composition $\jmath$ of $\imath$ with the hyperelliptic involution on $X_\eta$,
i.e.~it is induced from a section $P'$ on the resolved quotient 
\[
X'=\widetilde{X/\jmath}.
\]
This is yet another K3 surface which happens to be the quadratic twist of $\Jac(f)$ at the ramified fibers.
Assume that $\Jac(f)$ is given by a minimal Weierstrass form
\begin{eqnarray*}
\label{eq:jac}
\Jac(f):\;\;\; y^2 = x^3 + a_2x^2 + a_4x+a_6, \;\; a_i\in K[t],\;\; \deg a_i\leq i,
\end{eqnarray*}
not all $a_i$ being  scalar multiples of $i$-th powers of a single linear form in $t$.
%By a M"obius transformation, we can locate the ramified fibers at $t=0,\infty$.
Then the quadratic twist $X'$ may be defined by
\begin{eqnarray}
\label{eq:X'}
X': \;\;\; gy^2 =x^3 + a_2x^2 + a_4x+a_6
\end{eqnarray}
where the quadratic polynomial $g\in K[t]$ encodes the location of the ramified fibers
(i.e.~it has two distinct zeroes, possibly including $\infty$ so that it would be linear in that case).
Let us check under what condition a section $P'$ on $X'$ will lead to $\tau$ being fixed point free.
Clearly, $P'$ has to be disjoint from the induced zero section $O'$,
but there is a little more to it.
Namely, in order for $P$ not to meet $O$ at the ramified fibers,
$P'$ has to meet the ramified fibers of $X'$
(of Kodaira type $I_n^*$ for some $n\in\N_0$)
at a component different from $O'$,
and in fact in one of the two far components if $n>0$.

Writing $P'$ componentwise as $P' = (\tx,\ty)$ for a priori rational functions in $t$,
all the above translates as
\[
\deg \tx = \deg \ty = 2
\]
in terms of the twisted Weierstrass equation \eqref{eq:X'},
with extra conditions entering if far components of the ramified fibers are met
or if other singular fibers are intersected non-trivially (i.e.~at different components than $O'$ as will happen frequently in our setting).

In what follows, we shall start with the rational elliptic surface $\Jac(f)$
and aim for endowing its quadratic twist $X'$ with a suitable section $P'$.
Note that this is a codimension one condition in moduli
unless $P'$ is the classical example of a two-torsion section (but in that case there will be extra ramification conditions
cutting down the number of parameters in our setting).

\section{Extremal rational elliptic surfaces}
\label{s:res}

In order to attack determining the Enriques surfaces supporting the root types from \ref{ss:31},
and thus to prove Theorem \ref{thm1},
we turn our attention to Observation \ref{obs} \eqref{it1}.
Extremal rational elliptic surfaces have been studied starting from Beauville (for the semi-stable case over $\C$, cf.~\cite{Beauville}).
In positive characteristic, the picture changes only slightly, due to Lang \cite{Lang1}, \cite{Lang2}.
More precisely, the singular fiber types may change (or the surfaces may cease to exist anymore, such as $X_{3333}$
in characteristic $3$ as we have used in \ref{ss:no-add}; compare Table \ref{T0}),
but the torsion structure of the Mordell--Weil group does not change.

The following table reproduces the classification, following the notation over $\C$ from \cite{MP}.
The Weierstrass forms are designed for our purposes with a view towards quadratic twists;
hence, in order to include equations valid in characteristic $2$, major modifications would be in order.
 Note that the
 $2$-torsion sections are always located at zeros of $f(x)$;
 here we made sure to always normalise one of them, if any over $K(t)$,  to $(0,0)$.

 \begin{table}[ht!]
 $$
 \begin{array}{|ccccc|}
 \hline
 \text{notation} & \text{Weierstrass eqn. } y^2 = f(x) & \text{sing. fibers} & \text{char}(K) & \MW\\
 \hline
 \hline
 X_{9111} & x^3+(tx+1)^2/4
 & I_9/\infty, I_1/t^3-27 & p\neq 3 
 & \Z/3\Z\\
 && I_9/\infty, II/0 & p=3 & \\
 X_{8211} & x(x^2+(t^2+2)x+1) & I_8/\infty, I_2/0, I_1/t^2+4 & p\neq 2 & \Z/4\Z\\
 X_{6321} & x(x^2+(t^2/4+t-2)x+1-t) & %((1/4)*t^2+t-2)*x^2+x^3+(1-t)*x
 I_6/\infty, I_3/0, I_2/1, I_1/-8 & p\neq 2,3&
 \Z/6\Z\\
 &&  I_6/\infty, I_3/0, III/1 & p=3 &\\
 X_{5511} & x^3-tx^2+((1-t)x-t)^2/4 & I_5/0,\infty, I_1/t^2-11t-1 & p\neq 5 & \Z/5\Z\\
 && I_5/0,\infty, II/3 & p=5 &\\
 X_{4422} & x(x-1)(x-t^2) & I_4/0,\infty, I_2/\pm 1 & p\neq 2 & \Z/4\Z\times\Z/2\Z\\
 X_{3333} &
 x^3 + (3tx+t^3-1)^2/4
 & I_3/\infty, t^3-1 & p\neq 3 & (\Z/3\Z)^2\\
% X_{431} & x^3 + (x+t)^2/4 & IV^*/\infty, I_3/0, I_1/\frac 1{27} & p\neq 3 & \Z/3\Z\\
% && IV^*/\infty, I_3/0 & p=3 &\\
 %X_{411} & x(x^2+tx+t^4) & I_4^*/0, I_1/\pm \frac12 & \Z/2\Z\\
  X_{321} & x(x^2+x+t) & III^*/\infty, I_2/0, I_1/\frac14 & p\neq 2 & \Z/2\Z\\
 X_{222} & x(x-1)(x-t) & I_2^*/\infty, I_2/0,1 & p\neq 2 & (\Z/2\Z)^2\\
% X_{211} & x^3+x^2+t & II^*/\infty, I_1/0, -\frac 4{27} & p\neq 2,3 & 0\\
% && II^*/\infty, I_1/0 & p=3 &\\
 X_{141} & x(x^2+t(t+2)x+t^2) & I_4/\infty, I_1^*/0, I_1/-4 & p\neq 2 & \Z/4\Z\\
 X_{\lambda} & x(x^2+tx+\lambda t^2) \;\;\; (\lambda\in K^\times) & I_0^*/0,\infty & p\neq 2 & (\Z/2\Z)^2\\
 \hline
 \end{array}
 $$
 \caption{Extremal rational elliptic surfaces}
 \label{T0}
 \end{table}
 
 There is one surface, $X_{33}$, missing from the above table
 %because they are not needed for our considerations ($X_{431}$)
 because it is a degeneration of $X_{321}$.
 Let us explain this briefly
 (a similar, but more abstract argument for $X_{431}$ and $X_{44}$ has appeared in \cite{RS}).
 Compared to  $X_{321}$ as given in Table \ref{T0},
 $X_{33}$ has the singular fibers of type $I_2, I_1$ replaced by a $III$ fiber.
 Both surfaces can be combined in an isotrivial family of rational elliptic surfaces:
 \[
 \mathcal X: \;\; y^2 = x(x^2+cx+t), \;\;\; c\in K.
 \]
 Here all surfaces at $c\neq 0$ are isomorphic to $X_{321}$ (since $K$ is algebraically closed)
 while $X_{33}$ sits at $c=0$.
 Recall that for our Enriques surfaces, we are going to apply quadratic base changes to extremal rational elliptic surfaces
 (Observation \ref{obs} \eqref{it1}).
 Hence the following claim will allow us to tacitly omit $X_{33}$ in the sequel. % (and same for $X_{44}$).

 \begin{claim}
The family of quadratic base changes of $X_{321}$ contains those of $X_{33}$ in its closure.
\end{claim}
 
 \begin{proof}
 Let us write the base change by 
 \begin{eqnarray}
 \label{eq:bc}
 t = \gamma^{-1} \frac qr
 \end{eqnarray}
 for some quadratic polynomials $q,r\in K[s]$ and $\gamma\in K^\times$, say
 (which, of course, we could still normalize, for instance using M\"obius transformations).
 Applying \eqref{eq:bc} to $X_{321}$, we can rescale coordinates by 
 \[
 (x,y) \mapsto (\gamma^{-1/2} x, \gamma^{-3/4}y)
 \]
 to arrive at the Weierstrass form
 \[
 y^2 = x \left(x^2 +\gamma^{1/2} x + \frac qr\right).
 \]
 Extending to $\gamma=0$, we land exactly in the family of quadratic base changes of $X_{33}$.
 \end{proof}

 \section{One-dimensional families}
 \label{s:fam}
 
 We now come to the root types from \ref{ss:31},
 aiming to prove their existence and eventually deduce 
Theorem \ref{thm}.
The arguments proceed mostly as in characteristic zero (except for a few degenerations in characteristics $3,5$
which we shall point out below). 
In particular, we have the same critical result for the root lattice $R_0$:

\begin{Proposition}
\label{prop:curves}
If the elliptic fibration arises from the data in Table \ref{T1},
then all curves representing $R_0$ are taken to fiber components
by a suitable composition of reflections. 
\end{Proposition}

The proof of the proposition consists of a tedious case-by-case analysis
for each root type covering all possible fiber configurations including multiplicities
(and often using the obstructions posed by Lemma \ref{lem:no}).
Since the arguments are combinatorial in nature, essentially only using elementary properties of root lattices and Weyl groups,
they completely parallel the characteristic zero setting (which was sketched to some extent in \cite{S-Q-hom}).
We omit the details for brevity (except for a quick pointer in \ref{ss}).

In particular, Proposition \ref{prop:curves} allows us to read off 
how the smooth rational bisection $B$ 
intersects the singular fibers (depending on the multiplicities).
This translates into the configuration of $P'$
and thus yields explicit conditions for the Weierstrass form in \eqref{eq:X'}.
Spelling out these conditions, we can solve directly for the resulting families (see e.g.~\ref{ss}, \ref{ss:2nd})
and obtain the following result (paralleling the characteristic zero case).

 \begin{Theorem}
 \label{thm1}
 Let $Y$ be an Enriques surface supporting a root lattice $R$ of rank $9$.
Then $Y$ appears in Table \ref{T2'} or \ref{T2}.
 \end{Theorem}
 
 The notation in Tables \ref{T2'} and  \ref{T2} should be clear
from the previous sections: For each root type, we list the Jacobian
of the elliptic fibration induced by the isotropic vector $E$ from Table \ref{T1},
and indicate the quadratic twist \eqref{eq:X'} by $g$ together with the $x$-coordinate $\tx(P')$ of the section $P'$
leading to the Enriques involution $\tau$ on the K3 cover $X$.

 \begin{table}[ht!]
 $$
 \begin{array}{|clccc|}
 \hline
\#& \text{root type} & {\Jac(|2E|)} & \text{section:} \, \tx(P') & \text{quadr.~twist:} \, g\\
 \hline
% A_9 
%& X_{9111} & (-s^4t+{s}^{5}+2{s}^{2})/(s^3-1) &\\% -{\frac {{s}^{4}t}{{s}^{3}-1}}+{\frac {{s}^{5}+2\,{s}^{2}}{{s}^{3}-1}} &\\
%&& \multicolumn{2}{r}{ {s}^{4} \left( s-1 \right)  \left( {s}^{2}+s+1 \right) {t}^{2}-2{s}^{2} ({s}^{6}+2{s}^{3}-1) t+{s}^{9}+5{s}^{6}+7{s}^{3}-1} \\
1 & A_9 & X_{8211} & \mu t -1 & (\mu t-1) (\mu t +\mu^2-1)\\
2 & A_8 + A_1 & X_{6321} & \mu t + 1 & (\mu t+1)(\mu t + (2\mu+1)^2)\\
& \text{ also } \#27 &&&
%A_8+A_1 & X_{8211} & -s^2(2st-s^2+3)/(s^2+1) &  2s(s^2+1)t-(s-1)^2(s+1)^2\\
%& X_{9111} & 
%\multicolumn{2}{l|}{-s(s^2t^2+12st+3(s^2+4s+16))/(4(s-4)(s+2)^2)}
%\\
%&&\multicolumn{2}{r|}{s^2(3s+4)t^2+2s(s^2+16s+16)t+11s^3+60s^2+144s+64}
\\
3 & A_7+A_2 & X_{6321} & \mu(t-1)
&
\mu t^2+4\mu (\mu+1) t-4(\mu+1)^2
\\
& \text{ also } \#28 &&&\\

\hline
\end{array}
$$
 \caption{Explicit data for the 31 maximal root types}
% (those not existing in char.~$3$ being starred)}
 \label{T2'}
\end{table}

 \begin{table}[ht!]
 $$
 \begin{array}{|clccc|}
 \hline
\#& \text{root type} & {\Jac(|2E|)} & \text{section:} \, \tx(P') & \text{quadr.~twist:} \, g\\
 \hline
%& & 1 & t-\lambda\\
4 & A_7+2A_1 & X_{4422} & \mu t & (\mu t-1)(t-\mu)\\
& \text{ also }\#27, \#29 &&&\\
%& X_{8211} 
%& \multicolumn{2}{r|}{ -16(t-\lambda)^2/(\lambda^2+4)^2
%\;\;\;\;\;\;\; 4(\lambda^2+8)t^2-4\lambda(\lambda^2+12)t-(\lambda^2-4)^2}
%\\
%& X_{8211} & 
%(t^2+2s+2)/(s^2-1)
%&
%t^2+2s+2
%\\
5 & 
A_6+A_2+A_1 & X_{6321} &
-(ut-4u^2+3u-1)/(u^2(u+1))
&\\
&&&\multicolumn{2}{r|}{((u^2+u)t-(u-1)^2)(ut-4u^2+3u-1)}\\
6 & 
A_5+A_4 & X_{6321} & (t-1)(\mu t-1) & (\mu t-1)(\mu(4\mu+1)t-(2\mu+1)^2)\\
7 &
A_5+A_3+A_1 & X_{4422} &
\mu t+1-\mu & 
(\mu t+1-\mu)(t+1-\mu)
\\
& \text{ also } \#30
&&&\\
8^* & A_5+2A_2 & X_{6321} & 27(t+u)(t-1)/((u-8)(u+1)^2) &
(t+u)(3(u+4)t-(u-2)^2)\\
& \text{ also } \#28 &&&\\
9^* & A_5+A_2+2A_1 & X_{6321} & 27(t-\lambda)^2/((\lambda-1)(\lambda+8)^2) &\\
&&&\multicolumn{2}{r|}{
3(\lambda+11)t^2-4(\lambda^2+13\lambda+4)t+4(\lambda-4)^2}
\\

10^{\dagger} & 2A_4+A_1 & X_{5511} &
\multicolumn{2}{l|}{-(s^5t^2-s^3(11s^2-15s+5)t-(s-1)^5)/(5s^2-5s+1)^2}
\\
&&&\multicolumn{2}{r|}
{s^2(4s-1)t^2-2(2s-1)(11s^2-11s+2)t-(4s-3)(s-1)^2}
\\
11 & A_4+A_3+2A_1 & X_{4422} &
4\mu(t+\mu) &
(t+\mu) %(t-1-4\mu)
(4\mu t-4\mu^2-1)\\
 12 & 3A_3 & X_{4422} & \mu t^2-\mu+1
& \mu t^2-\mu+1 \\
& \text{ also } \#16, \#19 &&& \\

13 & 
2A_3 + A_2 + A_1 & X_{4422} & (\mu t-\mu+1)^2 & 
(\mu t-\mu+2)(\mu t+t-\mu+1)
\\
14 & 2A_3+3A_1 & X_{4422} &
-(t^2+4\mu t+1)/(4\mu^2-1)&
t^2+4\mu t+1
\\
15^* &
A_3+3A_2 
%& X_{3333} &
%(t-1)(\lambda-t) &\\
%&& \multicolumn{2}{r|}{3(\lambda+1)^2t^2-(4\lambda^3+12\lambda^2+6\lambda+2)t+4\lambda^3-1}
%\\
& X_{3333} & \multicolumn{2}{l|}{(t-1)((s^3+6s^2+9s+3)t+s^3+3s^2-3)/s^2(2s+3)}\\
&&&\multicolumn{2}{r|}{((6s+3)(s+1)^2t^2+(8s^3+12s^2-6s-6)t+4s^3-6s+3)}\\

16 &
D_9 & X_{9111} & \mu & 
\mu^2t^2+2\mu t+4\mu^3+1
\\
17 &
D_8+A_1 & X_{8211} & \mu & \mu t^2 + (\mu+1)^2
\\
& \text{ also } \#27 & X_{8211} & 0 & t-\mu\\
18 &
% D_8+A_1 & X_{222} & t & t(t-\lambda)\\
 D_7+2A_1 & X_{222} & t+\mu & (t+\mu)(t+\mu-1)
 \\
 19 &
 D_6+A_3 & X_{222} & (\mu t+1-\mu)t & (\mu t+1-\mu)(\mu t+1)\\
 & \text{ also } \#16 &&&\\
 20 &
 D_6+A_2+A_1 & X_{222} & -4\lambda t(\lambda t-1)
 & (\lambda t-1)(4\lambda^2 t-4\lambda+1)
 \\
 21 &
% D_6+3A_1 & X_{222} & 0 & t(t-\lambda)\\
 D_6 + 3A_1 & X_{222} & 0 & t(t-\lambda)\\
% X_{222} & (t+\lambda^2)^2/4\lambda &
% t^2+2\lambda t +\lambda^2-4\lambda
% \\
% & \text{ also 2nd case } & X_{222} &
%(\mu^2(t-1)^2+2\mu t+2\mu+1)/4\mu &
%\mu^2(t-1)^2+2\mu t+2\mu+1\\
%% & X_{222} & (\mu^2t^2-2\mu(\mu-1)t+(\mu+1)^2)/4\mu
%% &\mu^2t^2-2\mu(\mu-1)t+(\mu+1)^2
% \\
 22 &
 D_5+A_4 & X_{5511} & \multicolumn{2}{r|}{ \mu\; \, \;\;\;\;\;\;\;\;\;\;\;
 (\mu+1)^2t^2-2\mu(3\mu+1)t+\mu^2(4\mu+1)}
 \\
 23 &
 D_5 + A_3 + A_1 
 %& X_{141} &
 %((\lambda+1)^2t+4\lambda^2)^2/16\lambda^2 & \\
%  &&\multicolumn{2}{r|}{(\lambda^2+6\lambda+1)(\lambda+1)^2t^2+8\lambda^2(\lambda^2+4\lambda-1)t+16\lambda^4}\\
 & X_{4422} & 1 & (t-1)(t-\lambda)
 \\
 24 &
 D_5+D_4 & X_{141} & \mu & (\mu+1)t^2+2\mu t + \mu^2\\
 25 &
 D_4+A_3+2A_1 & X_{4422} &\mu & t^2-\mu\\
 & \text{ also } \#18 & X_{4422} & 0 & t-\lambda\\
 26 &
  2D_4+A_1 & \multicolumn{2}{l}{\;X_{-4\mu(4\mu+1)} \;\;\;\;\;\;\;\;\;\;\;\;\;\;\;\;\; (t+\mu)^2} & t^2+(6\mu+1)t+\mu^2
 \\
% 2D_4+A_1 & X_\lambda & (t+\mu)^2 & t^2+(6\mu+1)t+\mu^2
% \\
% & \multicolumn{2}{l}{\;\;\;\;\;\;\;[\lambda = -4\mu(4\mu+1)]} % -16*l^2-4*l
%& \\
 27 &
 E_8+A_1 & X_{321} & 0 & t(t-\lambda)\\
 28 &
 E_7+A_2 & X_{6321} & 0 & t-\lambda\\
 29 &
 E_7+2A_1 & X_{222} & 0 & t(t-\lambda)\\
% E_7 + 2A_1 & X_{321} & (t-u-3u^2)^2/4u(2u+1)^2&\\
% &&\multicolumn{2}{r|}{t^2-2u(7u+3)t+u(9u+4)(1+3u)^2}
% \\
 30 & E_6+A_3 & X_{141} & (\mu t-1)t & (\mu t-1)(\mu(\mu+1) t-1)
 \\
 31^* &
 E_6+A_2+A_1 & X_{6321}  & 0 & (t-1)(t-\lambda)
 \\
 \hline
 \end{array}
 $$
 \caption{Explicit data for Enriques surfaces supporting the 31 maximal root types
 (those not existing in char.~$3$ being starred  resp. daggered  for char.~$5$)} % resp. daggered for char.~$5$)}
 \label{T2}
 \end{table}

Since the proof of Theorem \ref{thm1} follows  the same lines as in characteristic zero,
we shall only comment on the degenerate cases (compare \cite[\S 9]{S-Q-hom}).

%
%
% 
% 
% 
% 
%\subsection{Proof of Theorem \ref{thm1}}
%
% If $R_0$ is nodally primitive,
% then this determines $\Jac(|2E|)$ by Lemma \ref{lem:prim}.
% Conversely, the Enriques surface $Y$, with fibration $|2E|$,
% arises from the base change construction from $\Jac(|2E|)$ by Proposition \ref{prop:imath}.
% 
% The key advantage of this construction is that it turns out to work in  all characteristics simultaneously
% (except possibly at $p=3,5$).
% To this end, we translate the given data into explicit equations
% based on Sections \ref{s:eqns} and \ref{s:res}.
% As an illustration, we discuss two root types in great detail.
% 

 \subsection{$\boldsymbol{R=A_5+2A_2}$}
 \label{ss}
 
Given an Enriques surface $Y$ supporting the root type $R=A_5+2A_2$, we work with the primitive isotropic vector $E$
from Table \ref{T1}. Then $R_0=A_5+A_2+A_1$, and $\Jac(|2E|)=X_{6321}$ by \ref{ex2}.
%Lemma \ref{lem:prim} and Table \ref{T1}.

We now study the smooth rational bisection $B=\sigma(a_1)$ induced by the vertex $a_1\in A_2$
whose dual vector is involved in $E$.
By definition, $B$ meets the $\tilde A_1$ fiber in two different components (unless it is ramified),
and all other reducible fibers in a single component 
-- the one which is not contained in the support of the image of the orthogonal summands $A_5$ and~$A_2$ of $R_0$
under the reflection $\sigma$.
Hence $\sigma(A_5)$ is exactly supported on five components of the $I_6$ fiber,
and further reflections take each $(-2)$-divisor precisely to a fiber component (as stated in Proposition \ref{prop:curves}).
The same reasoning applies to $\sigma(A_2)\subset\tilde A_2$,
and we conclude that the possible intersection patterns are exactly those depicted in Figure \ref{Fig:R}.

In order to pour this set-up into equations as indicated in Section \ref{s:eqns},
we first assume that no reducible fiber is multiple.
Hence we aim for a quadratic twist $X'$ of $X_{6321}$
at 
some quadratic polynomial $g$ without roots at $t=0,1,\infty$
such that there is a section $P'$ compatible with the  intersection pattern of $B$.
That is, we require $P'$ to meet the non-identity component of the $\tilde A_1$ fiber at $t=1$;
on the Weierstrass model \eqref{eq:X'},  $P'$ thus goes through the node at $(x,y,t)=(0,0,1)$.
In terms of the coordinates of $P'=(\tx,\ty)$, this means
\[
(t-1) \mid \tx, \ty.
\]
Setting $\tx=(t-1)(\mu t+\lambda)$, the right-hand side of \eqref{eq:X'} evaluates as
\begin{eqnarray}
\label{eq:H}
(t-1)^2 H_{\lambda,\mu}(t)
\end{eqnarray}
for some degree 4 polynomial $H\in K[\lambda,\mu][t]$ when substituting $x=\tx$.
For the section $P'$ to lie on $X'$ then simply requires
that \eqref{eq:H} factors as $g\ty^2$,
i.e. $H$ admits a square factor (agreeing with $\ty$),
and the residual quadratic polynomial has to be exactly $g$ (which is not a square!).
Put differently, the non-zero polynomial $H$ determines $\ty$ and $g$ as soon as its discriminant  vanishes.
But the discriminant of $H$ with respect to $t$ is readily computed:
\[
\Delta(H,t) = \mu^6 (\lambda +1)^3(\lambda ^3-6\lambda ^2\mu-15\lambda \mu^2-8\mu^3-27\mu^2).
\]
If $\mu=0$, then $P'$ meets the $I_6$ fiber at $t=\infty$ at a non-identity component,
but $B$ meets twice the same component, contradiction.
Similarly, if $\lambda=-1$, then $P'$ meets the $I_3$ fiber at $t=0$ at a non-identity component,
giving the same kind of contradiction.
It thus remains to investigate the cubic factor
which we transform to the following shape in auxiliary parameters $u,v$ with $\mu=1/v, \lambda=u\mu$:
\begin{eqnarray}
\label{eq:uv}
u^3-6u^2-15u-27v-8.
\end{eqnarray}
If char$(K)\neq 3$, then solving \eqref{eq:uv} for $v$ directly leads to
\[
\tx = 27(t+u)(t-1)/((u-8)(u+1)^2), \;\;\;
g = (t+u)(3(u+4)t-(u-2)^2)
\]
as recorded in Table \ref{T2}.
In conclusion,  the base change construction from Section \ref{s:bc} produces
a one-dimensional family of Enriques surfaces, parametrized by $u$ (with a few exceptional values, see \ref{ss:2nd}),
supporting the root type $R$.

In contrast, if char$(K)=3$, then \eqref{eq:uv} factors as $(u+1)^3$, i.e.~$\lambda=-\mu=\nu^3$,
giving
\[
\tx = \mu (t-1)^2, \;\;\; g = (t-1)((\nu^2+\nu) t+2(\nu^2+\nu+1)).
\]
Note, however, that the fiber at $t=1$ degenerates to Kodaira type $III$ in characteristic $3$ (see Table \ref{T0}).
With $g$ as above, this would be multiple on $Y$, contradiction.
In other words, the involution $\tau$ on $X$ (which is still well-defined as before)
ceases to be fixed-point free,
and the base change construction involving $P', g$ etc. does not produce Enriques surfaces in characteristic $3$ anymore,
but rather Coble surfaces.

%
%Hence, also in characteristic $3$ we derive a one-dimensional family of Enriques surfaces, parametrized by $\nu$ this time,
%supporting the root type $R$. 
%
%

%fibers, make sense of $u=\infty$?

\subsection{Second moduli component}
\label{ss:2nd}

What may come as a surprise at first is that we are not yet done with the given root type, in fact.
Namely, we still have to take multiple fibers into account.
For instance, ramification at the $I_3$ fiber visibly occurs at $u=0,2$ within the above family,
and $\tau$ stays fixed-point free exactly at the  CM point $u=2$.
For the $I_2$ fiber, we similarly have to exclude $u=-1$ and $u=8$.
However, the situation changes drastically if we allow for (or, in fact, impose abstractly) ramification at the $I_6$ fiber:
here the section $P$ is forced to meet the opposite component $\Theta_6$ of the resulting $I_{12}$ fiber on $X$.
Hence the theory of Mordell--Weil lattices \cite{ShMW} shows that $P$ has height zero,
i.e.~it is a torsion section, and in fact of order $2$.
But then $P=(0,0)$ is already defined on $X_{6321}$,
so substitution into \eqref{eq:X'} returns zero identically.
Hence, imposing the quadratic base change to ramifiy at $\infty$ yields a second family of Enriques surfaces $Y'$  
originating from 
\begin{eqnarray}
\label{eq:2nd}
P'=(0,0), \;\;\; g=t-\lambda \;\;\; (\lambda\in K\setminus\{0,-8\}).
\end{eqnarray}
The main difference to the previous case is that $Y'$ even supports the root type $E_7+A_2$,
i.e. the primitive closure $R'$ of $R$ (as one easily verifies in Figure \ref{Fig:R}).
Obviously this construction works in any characteristic (it can even be adjusted to characteristic $2$
as we  explore in \cite{S-2}).

\begin{Remark}
We shall verify in Section \ref{s:comp}
that the two families of Enriques surfaces constructed (outside characteristic $3$)
are in fact distinct.
\end{Remark}

\subsection{Further degenerations in characteristic $3$}
\label{ss:3'}

We have already seen in \ref{ss} that the data under \#8 degenerate in characteristic $3$
(but by \ref{ss:2nd}, the given root type is still supported on the family filed under \#28).
There are three more root types which degenerate, and in fact cease to exist on Enriques surfaces in characteristic $3$.

For \#15, 
this should be clear, since by Criterion \ref{crit-e},
the root lattice $R_0=4A_2$ from Table \ref{T1} implies that $|2E|$ would have four fibers
of Kodaira type $I_3$.
However, such a surface does not exist in characteristic $3$
(%since there cannot be $3$-torsion in the Mordell--Weil group of length two
as we have already used in \ref{ss:no-add}, compare Table \ref{T0}).

Meanwhile for \#31, this follows directly from the set-up from Table \ref{T1} as we shall sketch briefly.
By Criterion \ref{crit-e}, $\Jac(|2E|)=X_{6321}$,
and either the $I_6$ fiber is ramified or the bisection $B$ meets it in opposite components.
Consider the sections $O,P$ on the K3 cover $X$ which $B$ splits into.

\begin{claim}
$P=Q$ where $Q=(0,0)$ denotes the two-torsion section.
\end{claim}

\begin{proof}
Otherwise we may compute the height pairing between $P$ and $Q$ following \cite{ShMW}.
Since both meet the $I_6$ fibers (or the $I_{12}$ fiber in case of ramification)
in the same components (opposite the identity component met by $O$),
we obtain
\[
\langle P,Q\rangle = 2 - (P.Q) - 3 -\hdots \leq -1.
\]
But since $Q$ is torsion, the pairing is zero identically, contradiction.
\end{proof}
It follows that $B$ (just like $P$) meets the $\tilde A_1$ fiber in both components
unless it ramifies.
Presently, however, the Kodaira type is $III$ by Table \ref{T0} (since we are in characteristic $3$),
so there cannot be ramification.
But then the orthogonal summand $A_1$ embedding into $\tilde A_1$ cannot be perpendicular to $B$,
contradiction.
(In characteristic $p>3$, we infer  ramification at the $I_2$ fiber which exactly gives the family in Table \ref{T2}.)

It remains to discuss \#9.
Here the isotropic vector $E$ from Table \ref{T1}
leads to $\Jac(|2E|)=X_{6321}$, again by Criterion \ref{crit-e},
with smooth rational bisection 
meeting each singular fiber in a single component.
Searching for the corresponding section $P'$ on the quadratic twist $X'$
simplifies substantially if we take into account that $X'$ (just like $X_{6321}$) admits a 2-torsion section.
With the given data, this means (just like in \cite[8.1.4]{S-Q-hom}) that 
\begin{itemize}
\item
either $\tx(P')$ is a square in $k[t]$,
\item
or $\tx(P')$ is a scalar multiple of $g$.
\end{itemize}
In the first case, write $\tx(P')=\mu(t-\lambda)^2$.
Substitution into \eqref{eq:X'} yields a degree 4 polynomial in $k[t]$
which we require to split off a square factor.
Therefore~its discriminant vanishes
which one computes as
\begin{eqnarray}
\label{eq:ml}
(\lambda^3\mu+15\lambda^2\mu+48\lambda\mu-64\mu-27)(\lambda^2\mu-1)^3\mu^2
\end{eqnarray}
We first investigate the vanishing of  the first factor. % of \eqref{eq:ml} has to vanish.
If char$(k)\neq 3$, then we can solve for $\mu$ and obtain the  solution displayed in Table \ref{T2}.
However, if char$(k)=3$, then the first factor degenerates to
$\mu(\lambda-1)^3$.
Here $\lambda=1$ implies that the fiber at $t=1$ is multiple;
but the fiber is additive of type $III$ by Table \ref{T0}, contradiction.

We now investigate  the other two factors of \eqref{eq:ml}.
If $\lambda^2\mu=1$, then $P'$ meets the fiber at $t=0$ non-trivially.
By assumption, the fiber thus has to ramify
which forces $\lambda=4$ and locates this special member in the above family
(except that in characteristic $3$ it would cause the fiber at $t=1$ to ramify as well,
leading to the same contradiction as above).
Finally, if $\mu=0$, then $P$ is automatically two-torsion,
meeting the fibers at $t=1$ and $\infty$ non-trivially.
Hence both fibers have to ramify (which is only possible outside characteristic $3$ as before).
This surface can be located in the above family at $\infty$
(as can be seen by parametrizing for the section to take the shape $\tx(P')=(\mu t+\lambda)^2$, for instance).

%
%If $\mu=0$, then $P'=(0,0)$ is two-torsion
%while for $\lambda^2\mu=1$, it meets the $I_3$ fiber non-trivially.
%Hence we infer that the first factor of \eqref{eq:ml} has to vanish.
%If char$(k)\neq 3$, then we can solve for $\mu$ and obtain the displayed solution.
%However, if char$(k)=3$, then the first factor degenerates to
%$\mu(\lambda-1)^3$.
%Here $\lambda=1$ implies that the fiber at $t=1$ is multiple;
%but the fiber is additive of type $III$ by Table \ref{T0}, contradiction.
%
In the second case where $\tx(P')$ encodes the ramified fibers,
we can directly solve for the residual degree 4 factor of the RHS of \eqref{eq:X'}
after substituting to form a perfect square -- this turns out to be impossible
unless $P'$ meets the fiber at $t=0$ non-trivially, i.e.~$\tx(P')$ has constant coefficient $1$.
But then, as $\tx(P')$ encodes the ramified fibers,
the fiber at $t=0$ cannot ramify,
so $B$ will meet different fiber components, contradiction.

In summary, for each of the above root types (\#9, 15, 31),
we find exactly the families in Table \ref{T2}, but only outside characteristic $3$.

\subsection{Degeneration in characteristic $5$}
\label{ss:5}

The analysis of the root type $R=2A_4+A_1$ is similar to the last case,
but considerably more involved since there are no torsion sections on $X'$ to assist us.
For starters, Criterion \ref{crit-e} applies to the data in Table \ref{T1}
to show that $\Jac(|2E|)=X_{5511}$ and every fiber is met in a single component by the smooth rational bisection $B$.
Write
\[
\tx(P') = \nu t^2 +\mu t+\lambda
\]
and substitute into \eqref{eq:X'}.
Then the RHS polynomial $r$ has to split of the square of a quadratic polynomial (which will give $\ty(P')$).
Hence not only does the discriminant $\Delta(r,t)$ of $r$ with respect to $t$ vanish,
but also its discriminant (with respect to either coefficient of $\tx(P')$).
In detail, we have
\[
\Delta(r,t) = \nu^5\lambda^5 r_1.
\]
Since the first two factors correspond to the section $P'$ meeting the fiber at $t=\infty$ resp.~$t=0$
non-trivially, our assumptions  force the respective fiber to ramify. 
This leads to two isolated Enriques surfaces outside characteristics $3, 5$ which 
will be located in the family to be developed below at $s=0$ resp.~$s=1$.

We  continue  with the discriminant of the third factor $r_1$
which reads
\[
\Delta(r_1,\lambda) = \nu^5(\mu-1)^5 r_2^2 r_3^3.
\]
We have already discussed the first factor while $\mu=1$ leads to $r_1$ splitting off another factor $\lambda^2$.
We claim that we can also ignore $r_3=0$ for our purposes.
This can be read off directly from the definition of the discriminant:
here there would be a triple root instead of two double roots.
Thus we have to determine when the remaining factor
\begin{eqnarray*}
r_2 &  = & \mu^5+55 \mu^4 \nu+960 \mu^3 \nu^2+5060 \mu^2 \nu^3-1920 \mu \nu^4+176 \nu^5\\
&& \;\;\; -625 \mu^2 \nu^2-7500 \mu \nu^3+2500 \nu^4+3125 \nu^3
\end{eqnarray*}
 vanishes. If char$(k)=5$, then $r_2=(\mu+\nu)^5$,
 and $\nu=-\mu$ subsequently implies $\lambda = \mu-1$ and ramification at $t=3$.
 But there the fiber has type $II$, so the involution $\tau$ will not be fixed point free on $X$.
 
 On the other hand, if char$(k)\neq 5$, then $r_2$ has degree $2$ in any homogenizing variable.
 That is, $r_2$ defines a hyperelliptic curve $C$, with branch points affinely given by $3125(4\mu-1)(\mu+6)^4$.
It follows that the curve $C$ is birational to $\PP^1$ over the prime field,
and parametrizing $C$ gives the data in Table \ref{T2}.
%This completes the proof of Theorem \ref{thm1}.
%\qed

%
%\subsection{Proof of Theorem \ref{thm:31}}
%\label{ss:pf31}
%
%We should like to highlight
%that Table \ref{T2} implies Theorem \ref{thm:31}
%as all root types stated there are realized by the base change construction
%with isotropic vector from Table \ref{T1} and $R_0$ nodally primitive.
%\qed
%

\subsection{}
\label{ss:c}

We emphasize that Theorem \ref{thm1} implies Theorem \ref{thm}
except for  part \textit{(ii)} on the number of moduli components.
This will be treated in the next section.

\section{Component analysis}
\label{s:comp}

The families of Enriques surfaces from Tables \ref{T2'} and \ref{T2},
and the way they were derived,
give us a clear picture of the moduli components:
they are determined by the possible fiber configurations, including multiplicities,
which are compatible with the data in Table \ref{T1}.
Once this classification is achieved (as recorded in Theorem \ref{thm1}), it only remains to rule out potential symmetries.

\begin{Proposition}
\label{prop:number}
If a root type $R$ admits more than one moduli components,
then $R$ can be found in Table \ref{T3}.
Moreover, all but one components support
some proper overlattice $\hat R$ of $R$.
\end{Proposition}

The following table lists the root types admitting more than one component
together with the configurations of reducible fibers (with multiplicities) and the overlattices $\hat R$
supported on the extra families.
Note that in characteristic $3$
the starred root type $R=A_5+2A_2$ has only one component
since the first fiber configuration degenerates
(as studied in \ref{ss}).

\begin{table}[ht!]
$$
\begin{array}{|cccc|cc|}
\hline
 & R & R_0& \text{fibers} & \hat R & \text{fibers}\\
\hline
a) &A_8+A_1 & A_7+A_1 & I_8, I_2 & E_8+A_1 & 2I_8, I_2\\
b) &A_7+2A_1 & 2A_3+2A_1 & I_4, I_4, I_2, I_2 & E_7+2A_1 & I_2^*, 2I_2, I_2\\
&&&& E_8+A_1 & III^*, 2I_2\\
c) & A_7 + A_2 & A_5+A_2+A_1 & I_6,I_3,I_2 & E_7 + A_2 & 2I_6, I_3, I_2\\
d) & A_5+A_3+A_1 & 2A_3+2A_1 & I_4,I_4,I_2,I_2 & E_6+A_3 & I_1^*, I_4\\
e) & A_5+2A_2 \; (*) & A_5+A_2+A_1 & I_6,I_3,I_2 & E_7+A_2 & 2I_6,I_3,I_2\\
%D_8+A_1 & D_8 & I_4^* & E_8+A_1 & I_4^*
f) & 3A_3 & 2A_3+2A_1 & I_4, I_4, I_2, I_2 & D_6+A_3 & 2I_4, I_4, I_2, I_2\\
&&&& D_9 & I_1^*, 2I_4\\
g) &D_8+A_1 & A_7+A_1 & I_8, I_2 & E_8+A_1 & 2I_8, I_2
\\
h) & D_6+A_3 & D_6+2A_1& I_2^*, I_2,I_2 & D_9 & I_4^*
\\
i) & D_4+A_3+2A_1 & 2A_3+2A_1& I_4,I_4,I_2,I_2 
& D_7+2A_1 & 2I_4,I_4,I_2,I_2\\
\hline
\end{array}
$$
\caption{Root types with more than one moduli components}
\label{T3}
\end{table}

\subsection{Proof of Proposition \ref{prop:number}}

Deriving the families underlying Table \ref{T3} is an easy exercise
along the lines of the previous sections. 
We point out one non-trivial step here: the root type $3A_3$ in fact leads
to four fiber configurations (including multiplicities), supported on one-dimensional families,
but then two turn out to be equivalent (as checked in \cite[\S 10.2]{S-Q-hom}).

It remains to verify that the families are indeed distinct (as opposed to the symmetry alluded to above).
Over fields of characteristic zero, this can be achieved as follows.
The smooth rational curves on the Enriques surface $Y$ equip the K3 cover $X$ with a rank 19-sublattice
\[
L\subset\NS(X)
\]
generated by fiber components, the two sections $O, P$ mapping to the smooth rational bisection $B$ and torsion sections.
By the moduli theory of complex K3 surfaces, it follows that generically $\rho(\mX)=19$.
Upon checking that $L$ embeds primitively into $\NS(\mX)$ (which amounts to verifying that the section mapping to $B$ does 
not yield a divisible point in $\MW(\mX)$), we can thus compute the discriminant (of $\NS(\mX)$) of the generic member of $\mX$
based on the formula from \cite[(22)]{SSh}.
As soon as these discriminants differ, the families have to be distinct.

\subsection{}
In positive characteristic, it is still easy to check that $L$ embeds primitively into $\NS(X)$,
but there is a little complication
since we could have $\rho(\mX)=22$,
i.e. the whole family of K3 surfaces could be supersingular
(as happens for the K3 surfaces with maximal root types from Theorem \ref{thm:K3}).
We resolve this problem by the following result.

\begin{Proposition}
\label{prop:ord}
For each of the above families, the general member (over $\bar\F_p$) is an ordinary K3 surface.
\end{Proposition}

\begin{proof}
Assume that the general member $X$ over $\bar\F_p$ is not ordinary,
i.e.~$X$ has height $h\geq 2$ including possible $\infty$ (the supersingular case).
If $h<\infty$, then
\[
\rho(X) \leq 22 - 2h \leq 18
\]
which is impossible since $L$ has already rank $19$. Hence $h=\infty$, and $X$ is supersingular.
Denote the Artin-invariant of $X$ by $\sigma$
such that
\[
\NS(X)^\vee /\NS(X) \cong (\Z/p\Z)^{2\sigma}.
\]
More precisely, $\NS(X)$ is a $p$-elementary lattice of length at least $4$
(since the supersingular K3 surface with Artin invariant $\sigma=1$ is just a point in moduli by \cite{Ogus}).
We continue to argue similar to the proof of  Proposition \ref{lem:no-add}.
We have
\[
4\leq 2\sigma = \text{$p$-length}(\NS(X)) \leq \text{$p$-length}(L)+\underbrace{\text{$p$-length}(L^\perp)}_{\leq\,\text{rank}(L^\perp)=3}.
\]
We now have to distinguish whether $p=3$ or $p>3$.
In the latter case, $L$ always has $p$-length zero, so we obtain a contradiction.
The same holds for $b)$ and $f)-i)$ in characteristic $3$ including the second family for $a)$.
For the remaining families,
we proceed by exhibiting an explicit ordinary member of the family.
Note that Proposition \ref{prop:ord} genuinely concerns K3 surfaces,
so we are free to consider those surfaces where the Enriques involution degenerates
(as will be quite convenient).
In fact, we shall specialize to singular K3 surfaces over $\Q$ (i.e.~with $\rho=20$) and then switch to the smooth reduction mod $3$:
%\begin{enumerate}
\item[$a)$]
At $\mu=-\frac 12$, the $I_3$ fiber in the first family becomes ramified.
Thus 
\[
h(P)=4-2\cdot \frac 56 - \frac 32 = \frac 56,
\]
and we obtain $\rho(X)=20$, $d=$ discr$(\NS(X))=-20$
by \cite[(11.22)]{SSh}.
\item[$c)$]
The first family has an additional section $Q$ at $\mu=-2$,
namely $Q$ is also induced from a section $Q'$ on $X'$
with $x(Q')=(-1+1/\sqrt 2)t+1$ and 
\[
h(Q)=4-2\cdot \frac 56 - 2\cdot \frac 23 = 1.
\]
Since the height of $P$ is unchanged,
\[
h(P)=4-2\cdot \frac 56 - 2\cdot \frac 12 = \frac 43,
\]
it remains to compute the pairing $\langle P,Q\rangle = 2\langle P',Q'\rangle$.
Since $P', Q'$ are disjoint and meet one $I_0^*$ fiber at the same non-identity component and the other at different ones,
while the $I_6$ fiber is meet at the same fiber component 
(for suitable sign choices in the $y$-coordinate),
the latter gives 
\[
\langle P',Q'\rangle = 2 - 1 - \frac 12 - \frac 56 = - \frac 13.
\]
For the Mordell-Weil lattice on $X$, we thus obtain
\[
\MWL(X) = 
\begin{pmatrix} \frac 43 & -\frac 23\\
-\frac 23 & 1
\end{pmatrix}
\]
of determinant $-8/3$. Hence $\rho(X)=20$ and $d=-32$, again by \cite[(11.22)]{SSh}.
\\
The second family attains a ramified $I_3$ fiber at $\lambda=0$
such that $\rho=20, d=-8$.
\item[$d)$]
The first family attains two multiple fibers at $\mu=0$ while $P$ becomes two-torsion.
This yields $\rho=20, d=-32$.\\
The second family attains a ramified $I_1$ fiber at $\mu=-\frac 12$
which $P$ meets non-trivially. Hence 
\[
h(P)=4-2\cdot \frac 54 - \frac 12 = 1,
\]
so that again $\rho=20, d=-32$.
\item[$e)$]
The second family agrees with that of $c)$.
%\end{enumerate}

In either case, it is easily verified that the reduction modulo $3$ is good.
The proof of Proposition \ref{prop:ord} will be completed 
%in the remaining cases
%(in characteristic $3$) 
by verifying the following statement:

\begin{claim}
For each of the above K3 surfaces over $\Q$, one has $\rho(X\otimes\bar \F_3)=20$.
\end{claim}
\textit{Proof.}
To see the claim, one can use a kind of complementary statement to \cite[Prop.~4.1]{S-dyn}.
Here we can argue with the Shioda--Inose structure following \cite{SI}
which relates $X$ to a product of two isogenous elliptic curves $E, E'$ with CM in $\mathcal O_K$
or an order of conductor $2$ for $K=\Q(\sqrt d)$.
Since $d$ is a square modulo $3$, the elliptic curves are ordinary,
and 
$\rho((E\times E')\otimes\bar\F_3)=4$ which yields the claim.
Alternatively, one can argue purely on the level of lattices that $\NS(X)$ does not admit an embedding
into the N\'eron--Severi lattice of a supersingular K3 surface % of Artin invariant $\sigma=1$
for any $p$ such that $d$ is a square modulo $p$.
%\end{proof}
\end{proof}

\subsection{}
It follows from Proposition \ref{prop:ord} that the generic member of the family $\mX$ has Picard number $19$
and the same discriminant as in characteristic zero (after checking that there is no extra divisibility in $\MW(\mX)$ 
in characteristic $p$). In consequence, the given families are all distinct, and Proposition \ref{prop:number} is proven.
\qed

\subsection{Conclusion}

This proves  part \textit{(ii)} of Theorem \ref{thm}.
In view of \ref{ss:c}, the proof of Theorem \ref{thm} is thus complete.
\qed

\subsection*{Acknowledgements}

Thanks to Toshiyuki Katsura for  helpful comments.


\begin{thebibliography}{99}

%\bibitem{A} Artin, M.: \emph{Supersingular $K3$ surfaces}, Ann.~scient.~\'Ec.~Norm.~Sup.~(4) {\bf 7} (1974), 543--568.
%

\bibitem{ASD} Artin, M., Swinnerton-Dyer, P.: \emph{The Shafarevich-Tate conjecture for pencils of elliptic curves on $K3$ surfaces}, Invent.~Math.~{\bf 20} (1973), 249--266.
%%
%%
%%\bibitem{badescu}  Badescu L.: \emph{Algebraic Surfaces}, Universitext, Springer-Verlag, New York, 2001. 
%%
%%
%\bibitem{Barth}  Barth, W.~P.:   \emph{Lectures on K3- and Enriques surfaces.} in: {\it Algebraic geometry, Sitges (Barcelona 1983)}, Springer Lecture Notes in Math. {\bf 1124}  (1985), 21--57. 
%%%
%%%\bibitem{barth2000}   Barth, W.~P.:  \emph{$S_5$-symmetric quintic surfaces}, notes/personal communication, Erlangen, 2000.

%
%\bibitem{BHPV}
%Barth, W., Hulek, K., Peters, C., van de Ven, A.:
%\emph{Compact complex surfaces}.
%Second edition,
%Erg.~der Math.~und ihrer Grenzgebiete,
%3.~Folge, Band {\bf 4}.~Springer (2004), Berlin.
%
%\bibitem{BP}
%Barth, W.,
%Peters, C,:
%\emph{Automorphisms of Enriques Surfaces},
%Invent, Math. {\bf 73} (1983), 383--412.


\bibitem{Beauville} Beauville, A.:
\emph{Les familles stables de courbes elliptiques sur $\PP^1$
admettant quatre fibres singuli\'eres},
C.R.~Acad.~Sci.~Paris {\bf 294} (1982), 657--660.


%\bibitem{BC}
%Blanc, J.,
%Cantat, S.:
%\emph{Dynamical degrees of birational transformations of projective surfaces},
%preprint (2013),
%arXiv: 1307.0361.



\bibitem{BM1} Bombieri, E., Mumford, D.:
\emph{Enriques' classification of surfaces in char.~$p$. II},
in: Baily, W.~L.~Jr., Shioda, T.~(eds.),
 \emph{Complex analysis and algebraic geometry},
 Iwanami Shoten, Tokyo (1977) 23--42.


\bibitem{BM2} Bombieri, E., Mumford, D.:
\emph{Enriques' classification of surfaces in char.~$p$. III},
Invent.~Math.~{\bf 35} (1976), 197--232.




%
%\bibitem{BS}
%Boissiere, S., Sarti, A.:
%\emph{Counting lines on surfaces},
%Ann. Scuola Norm. Sup. Pisa Cl. Sci. {\bf 5} (2007), 39--52.
%
%\bibitem{Cayley}
%Cayley, A.:
%\emph{On the triple tangent planes of surfaces of the third order}, 
%Cambridge and Dublin Math.~J.~{\bf 4} (1849), 118--138.

\bibitem{CS}
Cartwright, D., Steger, T.:
\emph{Enumeration of the 50 fake projective planes}, 
Comptes Rendus Math. {\bf 348} (2010), 11--13.

%
%\bibitem{Clebsch}
%Clebsch, A.:
%\emph{Zur Theorie der algebraischen Fl\"achen},
%Journal reine angew. Math. {\bf 58} (1861), 93--108. 
%
%
%\bibitem{Cox} Cox, D.~A.: \emph{Primes of the Form $x^2+ny^2:$ Fermat, Class Field Theory, and  
%Complex Multiplication}, Wiley Interscience
%(1989).

\bibitem{Deligne}
Deligne, P.:
\emph{Relevement de surfaces K3 en charact\'eristique nulle},
Lect.~Notes Math.~{\bf 868} (1981), 58--71.


%
%\bibitem{DK}
%Dolgachev, I.,
%Keum, J.H.:
%\emph{K3 surfaces with a symplectic automorphism of order 11},
%JEMS {\bf 11} (2009), 799--818.
%
%\bibitem{ES-2}
%Elkies, N. D., Sch\"utt, M.:
%\emph{Genus 1 fibrations on the supersingular K3 surface in characteristic 2 with Artin invariant 1},
%preprint (2012),
%arXiv: 1207.1239.
%
%
%\bibitem{EO}
%Esnault, H.,
%Oguiso, K.:
%\emph{Non-liftability of automorphism groups of a K3 surface in positive characteristic},
%preprint (2014),
%arXiv: 1406.2761.
%
%
%\bibitem{EOY}
%Esnault, H.,
%Oguiso, K.,
%Yu, X.:
%\emph{Automorphisms of elliptic K3 surfaces and Salem numbers of maximal degree},
%preprint (2014),
%arXiv: 1411.0769.
%
%
%\bibitem{ES}
%Esnault, H., 
%Srinivas, V.: 
%\emph{Algebraic versus topological entropy for surfaces over finite fields}, 
%Osaka J. Math. {\bf 50} (2013), 827--846.
%
%%
%
%%\bibitem{HT}
%%Harris, J.,
%%Tschinkel, Y.:
%%\emph{Rational points on quartics},
%%Duke Math.~J.~{\bf 104} (2000),  477--500. 


\bibitem{HS}  Hulek, K., Sch\"utt, M.: \emph{Enriques surfaces and Jacobian elliptic surfaces},
Math. Z. {\bf 268} (2011), 1025--1056.


%
%\bibitem{HSa}
%Hulek, K., Sch\"utt, M.:
%\emph{Arithmetic of singular Enriques surfaces},
%Algebra \& Number Theory {\bf 6} (2012), 195--230.
%
%

\bibitem{HK}
Hwang, D., Keum, JH:
\emph{The maximum number of singular points on rational homology projective planes}, J. Algebraic
Geom. 
{\bf 20} (2010), 495--523.


\bibitem{HKO}
Hwang, D, Keum, JH, Ohashi, H.:
\emph{Gorenstein $\Q$-homology projective planes},
Science China Math. {\bf 58} (2015), 501--512.

\bibitem{Illusie}
Illusie, L.:
\emph{Complexe de de RhamÐWitt et cohomologie cristalline}, Ann. Sci. \'Ecole Norm.
Sup. {\bf 12} (1979), 501--661.

%
%%
%%
%%\bibitem{KS} Katsura, T., Shioda, T.: \emph{On Fermat varieties}, Tohoku Math.~J.~(2) {\bf 31} (1979), no.~1, 97--115.
%%
%
%\bibitem{Keum-Viet}
%Keum, JH:
%\emph{The moduli space of $\Q$-homology projective planes with 5 quotient singular points}, Acta Math. Vietnamica {\bf 35}
%(2010), 79--89.

%%
%%%\bibitem{RK} Kloosterman, R.: \emph{Elliptic K3 surfaces with geometric Mordell--Weil rank 15},  Canad.~Math.~Bull.~{\bf 50}  (2007),  no.~2, 215--226.
%%
%\bibitem{K} Kodaira, K.:
%\emph{On compact analytic surfaces I-III},
%Ann.~of Math., {\bf 71} (1960), 111--152;
%{\bf 77} (1963), 563--626; {\bf 78} (1963), 1--40.
%


\bibitem{Kondo-Enriques} Kond\=o, S.: \emph{Enriques surfaces with finite automorphism groups.}, Japan. J. Math. (N.S.) {\bf 12} (1986), 191--282. 


\bibitem{Kondo-Aut}
S.~Kond\=o, 
\emph{Automorphisms of algebraic K3 surfaces which act trivially
on Picard groups}, 
J.~Math.~Soc.~Japan, {\bf 44}, No.~1 (1992), 75--98.

\bibitem{Lang1}
Lang, W.~E.:
\emph{Extremal rational elliptic surfaces in characteristic p. I. Beauville surfaces},
Math. Z. {\bf 207} (1991), 429--437. 

\bibitem{Lang2}
Lang, W.~E.:
\emph{Extremal rational elliptic surfaces in characteristic p. II. Surfaces with three or fewer singular fibres},
Ark. Mat. {\bf 32} (1994),  423--448. 


%%%
%%%
%%%% \bibitem{lang} Lang, W.~E.: \emph{Classical Godeaux Surface in Characteristic P,} Math. Ann.~{\bf 256} (1981), 419--427.
%%%
%%%

\bibitem{Martin}
Martin, G.:
\emph{Enriques surfaces with finite automorphism group in positive characteristic},
preprint (2017), arXiv: 1703.08419.

\bibitem{MP} Miranda, R., Persson, U.:
\emph{On Extremal Rational Elliptic Surfaces},
Math.~ Z.~{\bf 193} (1986), 537--558.

\bibitem{Mumford-surf}
Mumford, D.:
\emph{Enriques' classification of surfaces in char.~$p$. I},
in: Global analysis.
Princeton, University Press 1969.





%%
%%\bibitem{N}
%%Nikulin, V.V.:
%%\emph{Finite groups of automorphisms of K\"ahlerian K3 surfaces}, 
%%Trudy
%%Moskov. Mat. Obshch. {\bf 38} (1979), 75--137. English translation: Trans. Moscow Math.Soc. {\bf 38}
%%(1980), 71--135.
%%
%%%
\bibitem{Nikulin} Nikulin, V.~V.:
\emph{Integral symmetric bilinear forms and some of their applications},
Math.~USSR Izv.~{\bf 14}, No.~1 (1980), 103--167.
%
%\bibitem{Nishi} Nishiyama, K.-I.: \emph{The Jacobian fibrations on some K3 surfaces and their Mordell--Weil groups}, Japan.~J.~Math.~{\bf 22} (1996), 293--347.
%
%
%\bibitem{OS} Oguiso, K., Shioda, T.: \emph{The Mordell--Weil lattice of a rational elliptic surface}, Comment.~Math.~Univ.~St.~Pauli~{\bf 40} (1991), 83--99.
%
%
%%
\bibitem{Ogus} Ogus, A.:
\emph{Supersingular K3 crystals},
Journ\'ees de G\'eom\'etrie Alg\'ebrique de Rennes (Rennes 1978),
Vol.~II,  3--86, Ast\'erisque {\bf 64}, Soc.~Math.~France, Paris, 1979.

%
%\bibitem{PSS} Piatetski-Shapiro, I.~I., Shafarevich, I.~R.: \emph{Torelli's theorem for algebraic surfaces of type ${\rm K}3$}, Izv.~Akad.~Nauk SSSR Ser.~Mat.~{\bf 35} (1971), 530--572.

\bibitem{PY}
Prasad, G., Yeung, S.-K.:
\emph{Fake projective planes}, 
Invent. Math. {\bf 168} (2007), 321--370;
addendum:  Invent. Math. {\bf 182} (2010), 213--227.

%%
%%%\bibitem{reid}
%%%Reid, M.:
%%%\emph{Campedelli versus Godeaux}. 
%%%Problems in the theory of surfaces and their classification (Cortona, 1988), 
%%%Sympos. Math. {\bf XXXII}, 
%%%309--365.
%%%

\bibitem{RS}
Rams, S., Sch\"utt, M.:
\emph{On Enriques surfaces with four cusps},
preprint (2014), arXiv: 1404.3924v2.

%%%
%%%\bibitem{S-NS} Sch\"utt, M.: \emph{K3 surfaces of Picard rank 20 over $\Q$}, Algebra \& Number Theory {\bf 4} (2010), no.~3, 335--356.
%%%
%%%%
%%%%\bibitem[S10]{S-dmv} Sch\"utt, M.: \emph{Arithmetic of K3 surfaces}, to appear in Jahresbericht der DMV, preprint (2008), arXiv: 0808.1061.
%%%%
%%%\bibitem{S-quintic} Sch\"utt, M.: \emph{Quintic surfaces with maximum and other Picard numbers}, 
%%%Journal Math. Soc. Japan {\bf 63} (2011), 1187--1201. 
%%%
%
%%%\bibitem{SS} Sch\"utt, M., Shioda, T.:
%%%\emph{An interesting elliptic surface over an elliptic curve},
%%%Proc.~Jap.~Acad.~{\bf 83}, 3 (2007), 40--45. 
%
%
%\bibitem{S-Fields} Sch\"utt, M.:
% \emph{Fields of definition of singular K3 surfaces},
% Communications in Number Theory and Physics {\bf 1}, 2 (2007), 307--321.
%

\bibitem{S-Q-hom} Sch\"utt, M.:
 \emph{Moduli of Gorenstein $\Q$-homology projective planes},
preprint (2015),
arXiv: 1505.04163v2.


\bibitem{S-dyn}
Sch\"utt, M.:
\emph{Dynamics on supersingular K3 surfaces}
Comment.~Math.~Helv.~{\bf 91} (2016), 705--719.

\bibitem{S-2}
Sch\"utt, M.:
\emph{$Q_\ell$-cohomology projective planes and singular Enriques surfaces in characteristic two},
preprint (2017),
arXiv: 1703.10441.

\bibitem{S-nodal}
Sch\"utt, M.:
\emph{Divisibilities among nodal curves},
preprint (2017), arXiv: 1706.00570.




%
%%
%%

%\bibitem{RS}
%Rams, S., Sch\"utt, M.:
%\emph{64 lines on smooth quartic surfaces},
%preprint (2013), arXiv: 1212.3511v3.
%
%
%%\bibitem{RS} Rams, S., Sch\"utt, M.: \emph{On quartics with lines of the second kind}, 
%%preprint (2013), to appear in Adv. Geom., arXiv: 1303.1304.
%
%\bibitem{Salmon}
%Salmon, G.:
%\emph{On the triple tangent planes to a surface of the third order,} 
%Cambridge and Dublin Math.~J.~{\bf 4} (1849), 252--260.
%
%%
%%\bibitem{schur} Schur, F.: \emph{Ueber eine besondre Classe von Fl\"achen vierter Ordnung.} 
%%Math. Ann. {\bf 20} (1882), 254--296.
%
%\bibitem{S-sigma=1}
%Sch\"utt, M.:
%\emph{A note on the supersingular K3 surface of Artin invariant 1},
%Journal of Pure and Applied Algebra {\bf 216} (2012), 1438--1441. 
%
%\bibitem{SS}
%Sch\"utt, M., 
%Schweizer, A.:
%\emph{On the uniqueness of K3 surfaces with maximal singular fibre}, 
%Annales de l'institut Fourier {\bf 63} (2013), 689--713. 
%
\bibitem{SSh} 
 Sch\"utt, M., Shioda, T.:
 \emph{Elliptic surfaces},
Algebraic geometry in East Asia - Seoul 2008, 
Advanced Studies in Pure Math.~{\bf 60} (2010), 51-160.  
%
%%\bibitem{SSvL} Sch\"utt, M., Shioda, T., van Luijk, R.: \emph{Lines on Fermat surfaces}, 
%%J.~Number Theory {\bf 130} (2010), 1939--1963.
%
%%%
%%%
%%%
%%%\bibitem{vGS}
%%%Sch\"utt, M.,
%%%van Geemen, B.:
%%%\emph{Two moduli spaces of abelian fourfolds with an automorphism of order five},
%%%to appear in Int.~J.~Math.,
%%% arXiv: 1010.3897.
%% 
%
%%\bibitem{Segre}
%%Segre, B.:
%%\emph{The maximum number of lines lying on a quartic surface}, 
%%Quart. J. Math., Oxford
%%Ser. {\bf 14} (1943), 86--96.
%%
%%\bibitem{Segre2}
%%Segre, B:
%%\emph{Forme e geometrie hermitiane, con particolare riguardo al caso finito},
%%Ann.~Mat.~Pura Appl.~(4) {\bf 70} (1965), 1--201.
%%
%%\bibitem{Sengupta}
%%Sengupta, T.:
%%\emph{Elliptic Fibrations on Supersingular K3 Surface with Artin Invariant 1 in characteristic 3},
%%preprint (2012), 
%%arXiv: 1204.6478.
%%
%
%
 \bibitem{Shimada-rk21} 
 Shimada, I.:  
\emph{Rational double points on supersingular $K3$ surface}, 
Mathematic of computation {\bf 73}, No. 248 (2004), 1989--2017.

%
% \bibitem{Shimada} Shimada, I.:  
%\emph{Transcendental lattices and supersingular reduction lattices of a singular $K3$ surface}, 
%Trans AMS {\bf 361} (2009), 909--949.



\bibitem{ShEMS} Shioda, T.:
 \emph{On elliptic modular surfaces},
 J.~Math.~Soc.~Japan {\bf 24} (1972), 20--59.

%
%%\bibitem{Sh-int}
%%Shioda, T.:
%%\emph{An example of unirational surfaces in characteristic p},
%%Math.~Ann.~{\bf 211} (1974), 233--236. 
%
%%
%%\bibitem{Sh}
%%Shioda, T.:
%%\emph{An explicit algorithm for computing the Picard number of certain algebraic surfaces},
%%Amer.~J.~Math.~{\bf 108}, No. 2 (1986), 415--432.
%% 
%%
%%
\bibitem{ShMW} Shioda, T.: \emph{On the Mordell--Weil lattices}, Comm.~Math.~Univ.~St.~Pauli {\bf 39} (1990), 211--240.
%
%\bibitem{Sh-Murre}
%Shioda, T.:
%\emph{Correspondence of elliptic curves and Mordell--Weil lattices of certain elliptic K3 surfaces},
%in \emph{Algebraic Cycles and Motives}, Vol.~2, 
%LMS Lect.~Note Ser.~{\bf 344}, Cambridge Univ. Press (2007),  319--339.
%
%%
%%% \bibitem{shioda}
%%%Shioda, T.:
%%%{\it The Mordell--Weil lattice of $y^2=x^3 + t^5 - 1/t^5 -11$}, 
%%%Comment. Math. Univ. St. Pauli {\bf 56} (2007), 45--70.
%
%
\bibitem{SI} Shioda, T., Inose, H.: \emph{On Singular $K3$ Surfaces}, in: W. L. Baily Jr., T. Shioda (eds.), \emph{Complex analysis and algebraic geometry}, Iwanami Shoten, Tokyo (1977), 119--136.
%
%\bibitem{SM} Shioda, T., Mitani, N.:
% \emph{Singular abelian surfaces and binary quadratic forms},
% in: \emph{Classification of algebraic varieties
% and compact complex manifolds},
% Lect.~Notes in Math.~{\bf 412} (1974), 259--287.
% 
%
%
%%%
%%
%%
%\bibitem{Tate} Tate, J.: {\it Algorithm for determining the type
%of a singular fibre in an elliptic pencil}, in: {\it Modular
%functions of one variable IV} (Antwerpen 1972), Lect.~Notes in Math.~{\bf 476}
%(1975), 33--52.
%
%\bibitem{Tate-C} Tate, J.:
% {\it Algebraic cycles and poles of zeta functions},
%in: {\it Arithmetical Algebraic Geometry}
%(Proc. Conf. Purdue Univ., 1963), 93--110, Harper \& Row (1965).
%
%%\bibitem{T1} Tate, J.: \emph{On the conjectures of Birch and Swinnerton-Dxer and a geometric analog}, in: A. Grothendieck, N. H. Kuiper (Hrsg.), \emph{Dix expos\'es sur la cohomologie des schemas}, 189-214. North-Holland Publ., Amsterdam, 1968.

%%%\bibitem{T2} Tate, J.: \emph{Conjectures on algebraic cycles in $\ell$-adic cohomology}, in: Proc.~Sympos.~Pure Math., Vol.~{\bf 55}, 71-83. Providence, RI: Amer.~Math.~Soc., 1994.
%%
%% \bibitem{vL} van Luijk, R.: \emph{K3 surfaces with Picard number one and infinitely many rational points},  Algebra \& Number Theory  {\bf 1} (2007),  no.~1, 1--15.

%\bibitem{Voloch}
%Voloch, F.:
%\emph{Surfaces in $\PP^3$ over finite fields},
%Topics in Algebraic and Noncommutative Geometry,
%Contemp.~Math.~{\bf 324} (2003), 219--226.
%
%%
%
%\bibitem{xiao}
%Xiao, G.:
%\emph{Galois covers between K3 surfaces},
%Annales de l'institut Fourier {\bf 46} (1996), 73--88. 

%%\bibitem{xie2010} Xie, Jinjing: \emph{More Quintic Surfaces with 75 Lines},
%%Rocky Mountain J. Math. {\bf 40} (2010), 2063--2089.

\bibitem{Ye}
 Ye, Q.:
 \emph{On Gorenstein log del Pezzo surfaces}, 
 Japan J. Math. {\bf 28} (2002), 87--136.

%%

\end{thebibliography}
\end{document}